\documentclass{siamart1116}

\usepackage{latexsym,mathrsfs,multirow}
\usepackage{amsmath,amssymb}
\usepackage{enumerate,verbatim}
\usepackage{amsfonts}
\usepackage{graphicx}
\usepackage{algorithm}
\usepackage{algorithmic}
\usepackage{url}
\usepackage{amsopn}
\usepackage{lipsum}
\usepackage{booktabs}
\usepackage{multirow}

\usepackage{cleveref}
\usepackage{float}

\newtheorem{example}{Example}

\newtheorem{remark}{Remark}
\newtheorem*{problem*}{Problem}

\numberwithin{equation}{section}
\numberwithin{table}{section}
\numberwithin{figure}{section}
\numberwithin{theorem}{section}

\DeclareMathOperator*{\argmin}{argmin}

\newcommand{\nrm}[1]{{\left\vert\kern-0.25ex\left\vert\kern-0.25ex\left\vert #1 
		\right\vert\kern-0.25ex\right\vert\kern-0.25ex\right\vert}}

\def \R{{\mathbb R}}
\def \C{{\mathbb C}}
\def \bbS{\mathbb S}

\def \dA{\Delta A}
\def \dB{\Delta B}
\def \dC{\Delta C}
\def \dP{\Delta P}
\def \dS{\Delta S}
\def \dpS{\Delta^{\!\!\Pi}\! S}

\def \pset{\mathbb U}
\def \ppset{\mathbb U^{\Pi}}

\def\rhom{\rho_{\!\!_{\mathcal M}}}

\ifpdf
  \DeclareGraphicsExtensions{.eps,.pdf,.png,.jpg}
\else
  \DeclareGraphicsExtensions{.eps}
\fi

\begin{document}

\title{Eigenvalue backward errors of Rosenbrock systems
and optimization of sums of Rayleigh quotient
}

\renewcommand{\thefootnote}{\fnsymbol{footnote}}

\author{
Ding Lu\thanks{%
               Department of Mathematics, 
			   University of Kentucky, 
			   Lexington, KY-40506 
			   (\email{ding.lu@uky.edu}).
				D.L. is supported in part by NSF DMS-2110731.
		   }
\and
Anshul Prajapati\thanks{%
				Department of Mathematics, 
				Indian Institute of Technology Delhi, 
				Hauz Khas, 110016, India 
				(\email{maz198078@maths.iitd.ac.in}, 
				\email{punit.sharma@maths.iitd.ac.in}).
				A.P. acknowledges the support of the CSIR Ph.D. grant by Ministry of Science \& Technology, Government of India.
				P.S. acknowledges the support of SERB MATRICS Project (MTR/2022/000362), Government of India.
			}
\and
Punit Sharma\footnotemark[3]
\and
Shreemayee Bora\thanks{%
                Department of Mathematics,
                Indian Institute of Technology Guwahati,
                Guwahati-781039, India
                (\email{shbora@iitg.ernet.in}).
				S.B. is supported by SERB MATRICS Project (MTR/2019/000383), Department of Science and Technology (DST), Government of India.
                }
}

\headers{Eigenvalue backward errors of Rosenbrock systems}{Ding Lu, Anshul Prajapati, Punit Sharma, and Shreemayee Bora}

\maketitle

\begin{abstract}
%
%
We address the problem of computing the eigenvalue backward error
of the Rosenbrock system matrix under various types of block perturbations.
We establish computable formulas for these backward errors 
using a class of minimization problems involving the 
Sum of Two generalized Rayleigh Quotients (SRQ2). 
For computational purposes and analysis, we reformulate such optimization problems 
as minimization of a rational function over the
joint numerical range of three Hermitian matrices. 
This reformulation eliminates certain local minimizers of the
original SRQ2 minimization and allows for 
convenient visualization of the solution.
Furthermore, by exploiting the convexity within the joint numerical range,
we derive a characterization of the optimal solution 
using a Nonlinear Eigenvalue Problem with Eigenvector dependency (NEPv).
The NEPv characterization enables a more efficient solution of the 
SRQ2 minimization compared to traditional optimization techniques. 
Our numerical experiments demonstrate the benefits and effectiveness of 
the NEPv approach for SRQ2 minimization in computing eigenvalue backward
errors of Rosenbrock systems.
\end{abstract}





\section{Introduction}
Consider the {\em Rosenbrock system matrix}
(also known as the Rosenbrock system polynomial) in the standard
form~\cite{MR3477318,MR4742107,MR0325201}:
\begin{equation}\label{eq:rosenbrock}
	S(z)=\begin{bmatrix} A-z I_r & B \\ C & P(z) \end{bmatrix},
\end{equation}
where $z\in\C$, $A\in \C^{r, r}$, $B\in \C^{r, n}, C\in \C^{n, r}$,
$I_r$ is an identity matrix of size $r$, 
and $P(z)$ is a matrix polynomial of degree $d$
given by $P(z)=\sum_{j=0}^d z^jA_j $ with $A_j\in \C^{n, n}$ for $j=0,\ldots,d$.
A scalar $\lambda \in \C$ is called an eigenvalue of $S(z)$ 
if $\det(S(\lambda))=0$. 


%
The Rosenbrock system matrix~\eqref{eq:rosenbrock} and its
associated eigenvalues are of fundamental importance in the
field of linear system theory. 
It is well-known that the dynamical behaviour of 
a linear time-invariant system in the form of 
\begin{equation}\label{eq:ltis}
	\Sigma : \quad
	\left\{
	\begin{array}{rl}
		\dot{x}(t)&=Ax(t)+Bu(t),\\
		y(t)&=Cx(t)+P(\frac{d}{dt})u(t)
	\end{array}\right.
\end{equation} 
can be analyzed through the system matrix $S(z)$ in~\eqref{eq:rosenbrock},
as it contains the information on the {\em poles and zeros} of the
transfer function of $\Sigma$, which determine the stability and performance of the system.
The eigenvalues of $S(z)$ are known as the {\em invariant zeros} of
$\Sigma$ and can be used to recover the zeros of the system.  
For more details, refer to~\cite{MR0325201}.

The Rosenbrock system matrix is also found in solving rational eigenvalue problems,
where for a given a rational matrix function $R(z)\in\C^{n\times n}$, 
we seek scalars $\lambda\in\C$ such that $\det(R(\lambda))=0$.
Rational eigenvalue problems constitute
an important class of nonlinear eigenvalue problems
with numerous applications, including vibration analysis
and solution of (genuine) nonlinear eigenvalue problems; 
see, e.g.,~\cite{MR3031626,Guttel:2017,MR2124762,MR1988722}. 
A popular method for solving rational eigenvalue problems
is the {\em linearization} of $R(z)$
through its {\em state-space realization}, expressed as
\begin{equation}\label{eq:rep}
	R(z)=P(z) - C(A-zI_r)^{-1}B,
\end{equation}
where $A\in \C^{r,r}, B\in \C^{r,n}, C\in \C^{n,r}$, 
and $P(z)$ is matrix polynomial of degree $d$ and size $n$.
This rational linearization technique was introduced by~\cite{MR2811297}
and further developed 
in~\cite{MR3477318,MR3878309,MR3908736,MR3910502,MR4576817,MR4742107}.
The construction and analysis of the linearization rely on 
the fact that, under mild assumptions, 
the eigenvalues of $R(z)$ given by~\eqref{eq:rep} can be recovered 
from the eigenvalues of the Rosenbrock system matrix~\eqref{eq:rosenbrock}.

The main purpose of this paper is to perform an eigenvalue 
backward error analysis for the Rosenbrock system matrix~\eqref{eq:rosenbrock},
taking into account both full and partial perturbations 
across its four component blocks: $A$, $B$, $C$, and $P(z)$.
The eigenvalue backward error is a fundamental tool in both the
theoretical study and practical application of numerical 
eigenvalue problems~\cite{Stewart:1990}. 
As a metric for the accuracy of approximate eigenvalues, 
backward error is commonly 
applied in the stability analysis of eigensolvers. 
Moreover, the backward perturbation matrix 
can also provide useful insights into the sensitivity of eigenvalues to
perturbation in the coefficient matrices.

The eigenvalue backward error for matrix polynomials
have been extensively studied in the literature; 
see, e.g.,~\cite{MR2496422,MR2780396,MR4098788,MR3194659,MR3335496,MR4404572,Tis00}. 
A recurring theme in those studies is the preservation 
of particular structure within the perturbation matrices,
such as symmetry, skew-symmetry, and sparsity.
However, these works have not yet explored the unique block 
structure in the Rosenbrock system matrix $S(z)$ given
by~\eqref{eq:rosenbrock}, despite the prevalence of such matrix polynomials.

Preserving the block structures within $S(z)$ 
for backward error analysis holds practical significance.
For instance, in the linear system~\eqref{eq:ltis}, 
certain coefficient matrices may be exempt from perturbation due to 
system configurations, 
and a block-structured perturbation to the Rosenbrock system $S(z)$ 
can best reflect such constraints.
For rational eigenvalue problems with $R(z)$ given by~\eqref{eq:rep}, 
maintaining the block structure in the perturbation matrix to $S(z)$
allows for the derivation of eigenvalue backward errors 
relating to the coefficient matrices of the rational matrix function $R(z)$,
by exploiting equivalency of the two eigenvalue problems.
Currently, there is limited research on backward error analysis for
rational eigenvalue problems.  
Notably, \cite{MR3568297}~discusses backward errors for eigenpairs,
and a recent study~\cite{MR4709341} addresses backward errors in the
context of perturbation matrix structures, such as symmetric, Hermitian,
alternating, and palindromic.

\subsection{Contribution and outline of the paper}
The main contribution of the paper is two-fold:
(i) to establish computable formulas for the eigenvalue backward errors
of the Rosenbrock system matrix under various types of perturbations 
to the blocks $A,B,C$, and $P(z)$ in $S(z)$, 
and (ii) to develop a novel approach based 
on nonlinear eigenvector problems
for the solution of related optimization problems.

The rest of the paper is organized as follows.
In~\Cref{preliminaries}, we introduce the backward errors
for an approximate eigenvalue $\lambda$ of 
a Rosenbrock system matrix $S(z)$ given by~\eqref{eq:rosenbrock}.
We consider various types of perturbations to the blocks
$A,B,C$, and $P(z)$ of $S(z)$, 
including fully perturbing all four blocks and
partially perturbing only one, two, or three of these blocks. 
We show that the computation of those backward errors 
will lead to a particular type of minimization problem over the unit sphere, 
where the objective function is the Sum of Two generalized Rayleigh
Quotients (SRQ2) of Hermitian matrices. 
Note that optimization problems involving Rayleigh quotients 
are commonly employed for structured backward error analysis 
of eigenvalue problems, 
see, e.g.,~\cite{MR3194659,MR3335496,MR2837582,MR4404572}.

In~\Cref{sec:SRQ2NEPv}, we focus on the SRQ2 minimization problem.
Our development is based on the observation that such problems can be
reformulated as minimization problems over
the joint numerical range (JNR) of three Hermitian matrices. 
This allows us to exploit the inherent convexity in the JNR 
for the solution of the problem. 
Particularly, we establish optimality condition in the form of 
Nonlinear Eigenvalue Problem with eigenvector dependency (NEPv).
This leads to a more efficient and reliable solution of SRQ2 minimization.
Moreover, the convex JNR also allows for convenient visualization of the
optimal solution and verification of global optimality.

In~\Cref{sec:example}, we present numerical examples of SRQ2 minimization, 
intended for computing eigenvalue backward errors of the Rosenbrock system matrix, 
and demonstrate the performance of NEPv approach compared 
to the state-of-the-art optimization technique for SRQ2 minimization. 
We conclude in~\Cref{sec:conclusion}.

\subsection{Notations}%
$\C^{m,n}$ (or $\R^{m,n}$) is the set of $m\times n$ complex (or real) matrices,
with $\C^{m}\equiv \C^{m,1}$ (or $\R^{m}\equiv \R^{m,1}$).
$I_k$ is the $k\times k$ identity matrix.
For a vector or matrix $M\in\C^{m, n}$, 
$M^{T}$ is the transpose, 
$M^{*}$ is the conjugate transpose,
$\|M\|$ is the $2$-norm (i.e, spectral norm),
and $\|M\|_F$ is the Frobenius norm.
We use 
\[\mathbb S^n:= \{x \in \C^n \colon \|x\|=1\}\]
to denote a unit sphere in the $2$-norm in $\C^n$.
For a Hermitian matrix $M$, 
$M\succeq 0$ (or $M\succ 0$) means that it 
is positive semidefinite (or definite).
We use $\lambda_{\min}(M)$ to denote the 
smallest eigenvalue of a Hermitian matrix $M$,
and $\lambda_{\min}(M,N)$ to denote the smallest eigenvalue
of a definite Hermitian matrix pencil $M-\lambda N$.
We will use frequently generalized Rayleigh quotients 
$\rho(x) := \frac{x^*Gx}{x^*Hx}$, with $x\in\C^n\setminus\{0\}$,
for positive semidefinite $G$ and $H\in\C^{n,n}$,
where we assume that
\begin{equation}\label{eq:defrq}
	\rho(x)= \frac{x^*Gx}{x^*Hx} :=0, \quad\text{if $x^*Hx = 0$ and $x^*Gx=0$.}
\end{equation}
By the definition~\eqref{eq:defrq}, the function $\rho(x)$ is  well-defined 
and lower semi-continuous, i.e., $\liminf_{y\to x}\rho(y)\geq \rho(x)$,
for all $x\in\C^n\setminus\{0\}$,
so we can properly define optimization problems 
involving Rayleigh quotients in later discussions.
Since Rayleigh quotients are homogeneous,
i.e., $\rho(x)=\rho(cx)$ for $c\neq 0$,
we can restrict its variable to $x\in\mathbb S^n$.

\section{Backward Error Analysis of Rosenbrock Systems}\label{preliminaries}
We consider eigenvalue backward error of the Rosenbrock
system matrix~\eqref{eq:rosenbrock} subject to various types of
block perturbation in its coefficient matrices. 
The following result is well established in matrix analysis
and will be frequently used in our derivation of computable formulas
for the backward errors.

\begin{lemma}\label{lem:mapping}
	\!Let $x\in \C^n$\!, $b\in \C^m$\!, 
	$\Omega\!=\!\{D\in \C^{m,n} \colon D x\!=\!b\}$, and
	$\smash{\displaystyle \widehat D\!=\!\argmin_{D\in \Omega}\|D\|_F^2}$.

	\begin{enumerate}[(a)]
		\item \label{lem:mapping:ia}
		If $x\neq 0$,  then $\Omega \neq \emptyset$ 
		and the minimizer $\widehat D=bx^{*}/\|x\|^2$.
		\item \label{lem:mapping:ib}
		If $x=0$ and $b=0$, then $\Omega  = \C^{m,n}$ 
		and the minimizer $\widehat D =0$.
		\item \label{lem:mapping:ic}
		If $x=0$ and $b\neq 0$, then $\Omega  = \emptyset$ and
		$\widehat D$ has no solution.
	\end{enumerate}
\end{lemma}

\begin{proof}
	It follows from standard analyses in optimization 
	that the minimizer $\widehat D$ is unique,
	as the objective function $\|D\|_F^2$ is strongly convex and 
	$\Omega$ is a convex set; see, e.g.,~\cite{Nocedal:2006}.
	For all $D \in\Omega$, the constraint $b=D x$ leads to 
	$\|b\|\leq \|D\|_F\|x\|$. 
	For item~\eqref{lem:mapping:ia} with $x\neq 0$,
	we have $\|D\|_F\geq \|b\|/\|x\|$,
	where the equality holds for $\widehat D=bx^{*}/\|x\|^2$,
	which satisfies $\widehat Dx=b$.
	Items~\eqref{lem:mapping:ib} and~\eqref{lem:mapping:ic} follow from
	straightforward verification.
\end{proof}

\subsection{Backward Error with Full Perturbation}\label{full}
Let $\dS(z)$ be a perturbation to the Rosenbrock system matrix $S(z)$
in the form of 
\begin{equation}\label{eq:ds}
\dS(z):=\begin{bmatrix}\Delta A & \Delta B\\ \Delta C & \Delta P(z) \end{bmatrix} \;, 
\end{equation}
where 
$\Delta A \in \C^{r,r}, \Delta B \in \C^{r,n}, \Delta C \in \C^{n,r}$,
and $\Delta P(z)=\sum_{j=0}^d z^j \Delta {A_j}$,
with $\Delta {A_j} \in \C^{n,n}$ for $j=0,\ldots, d$,
representing perturbations to the corresponding coefficient matrices of
the Rosenbrock system $S(z)$ in~\eqref{eq:rosenbrock}.
The magnitude of the perturbation $\dS(z)$ is 
naturally measured by the Frobenius norms of its coefficient matrices,
i.e., 
\begin{equation}\label{eq:nrm}
	\nrm{\dS(z)}:=\sqrt{{\|\dA\|}_F^2+{\|\dB\|}_F^2+{\|\dC\|}_F^2+{\|\Delta {A_0}\|}_F^2+\cdots +{\|\Delta {A_d}\|}_F^2}.
\end{equation}
We introduce the following notion of eigenvalue backward error of
Rosenbrock systems.

\begin{definition}\label{def:error}
	The backward error of an approximate eigenvalue $\lambda\in\C$ 
	of the Rosenbrock system matrix $S(z)$ as from~\eqref{eq:rosenbrock},
	subject to the perturbation from the set 
	\begin{equation}\label{eq:bs}
		 \pset
		:= \{\dS(z)\,\colon\, \text{$\dS(z)$ is in the form
		of~\eqref{eq:ds}} \},
	\end{equation}
	is defined as the smallest $\dS(z)\in\pset$ measured by
	$\nrm{\cdot}$ from~\eqref{eq:nrm},
	such that $\lambda$ is an eigenvalue of the perturbed 
	Rosenbrock system matrix $S(z) - \dS(z)$, i.e.,
	\begin{equation}\label{eq:error}
		\eta(\lambda):=\min\left\{ \nrm{\dS(z)}
			: 
		\text{$\dS(z)\in\pset$, $\det(S(\lambda)-\dS(\lambda))=0$}
		\right\}.
	\end{equation}
\end{definition}
%

\begin{remark}\label{rmk:eta}
\rm 
First, if $\lambda$ is an eigenvalue of $S(z)$, 
then the backward error $\eta(\lambda)=0$ vanishes,
since $\dS(z)=0$ is a solution to~\eqref{eq:error}.
Second, for all $\lambda\in \C$,
we have $\eta(\lambda)$ bounded by $\eta(\lambda)\leq \nrm{S(z)}$,
by taking the feasible perturbation $\dS(z)\equiv S(z)$
in~\eqref{eq:error}.
Moreover, the minimal perturbation in~\eqref{eq:error} is always achievable,
i.e., there exists $\dS_{\star}(z)\in\pset$ such that $\eta(\lambda)=\nrm{\dS_{\star}(z)}$. 
This is because ${\det(S(\lambda)-\dS(\lambda))}$ is a polynomial in 
the coefficients of $\dS(z)$.
The set of feasible perturbations $\dS(z)$ of~\eqref{eq:error}, which is the zero set of the polynomial
as given by $\det(S(\lambda)-\dS(\lambda))=0$, must be a closed set. 
Consequently, 
in the minimization~\eqref{eq:error},
the continuous objective function $\nrm{\cdot}$
over a closed set of $\dS(z)$ can always 
achieve its minimal value. 
\end{remark}

Working directly with the constraint on the determinant of a matrix, as
in the formula for backward error in~\eqref{eq:error},
may pose computational challenges.
Instead, we can use the fact that a square matrix $M$ has 
$\det(M)=0$ if and only if $Mx=0$ for some nonzero vector $x$.
We can hence reformulate the optimization problem~\eqref{eq:error} as
\begin{align}
\eta(\lambda)&=\min\left\{ \nrm{\dS(z)} : 
\text{$\dS(z)\in\pset$, 
$x\in\bbS^{r+n}$, 
$(S(\lambda)-\dS(\lambda))x=0$} \right\} \notag \\
					 & = \min_{x\in\bbS^{r+n}} 
					 \bigg(
					 \min \left\{ \nrm{\dS(z)} : 
					 \text{$\dS(z)\in\pset$, $S(\lambda)x=\dS(\lambda)x$} \right\}
					 \bigg), 
					 \label{eq:error2}
\end{align}
where the second equation separates the variables $x$ and $\dS(z)$ 
into a bi-level optimization.
For the inner minimization problem on $\nrm{\dS(z)}$,
we can show that the minimal value can be expressed as the sum of 
two generalized Rayleigh quotients (SRQ2) in $x$.
Consequently, we obtain a characterization of $\eta(\lambda)$ 
through an SRQ2 minimization problem, as shown in the following result.

\begin{theorem}\label{thm:full}
	Let $S(z)$ be a Rosenbrock system matrix given by~\eqref{eq:rosenbrock}, 
	$\lambda \in \C$ with $\det(S(\lambda))\neq 0$,
	and the backward error $\eta(\lambda)$ be defined as in~\eqref{eq:error}.
	We have
	\begin{equation}\label{eq:allbloexp}
		\eta(\lambda)=\min_{x\in \bbS^{r+n}}
		\left(\frac{x^*G_1x}{x^*x}+\frac{x^*G_2x}{x^*(H_1+\gamma H_2)x} \right)^{1/2},
	\end{equation}
	where $\gamma:=\sum_{j=0}^d |\lambda|^{2j}$,
	and $G_i,H_i$ are Hermitian matrices of size $r+n$ given by 
	\begin{subequations} \label{eq:gmat} 
	\begin{align}
		G_1&:=\begin{bmatrix} (A-\lambda I_r) & B
			\end{bmatrix}^*\begin{bmatrix} (A-\lambda I_r) & B
		\end{bmatrix},\\ \label{eq:gg1}
			G_2&:=\begin{bmatrix} C & P(\lambda)\end{bmatrix}^*\begin{bmatrix} C &
		P(\lambda)\end{bmatrix}, \\ \label{eq:gg2}
		H_1&:=\begin{bmatrix}I_r & 0\\ 0& 0\end{bmatrix},\quad \text{and}\quad
		H_2:=\begin{bmatrix}0& 0\\ 0& I_n\end{bmatrix}.
	\end{align}
	\end{subequations}
\end{theorem}

\begin{proof}
	Consider the inner minimization problem of~\eqref{eq:error2}: 
	\begin{equation}\label{eq:inneropt}
		\min \left\{ \nrm{\dS(z)} : \text{$\dS(z)\in\pset$,
		$S(\lambda)x=\dS(\lambda)x$} \right\},
	\end{equation}
	where $x$ is a fixed vector.
	By the block form of coefficient matrices $\dS(z)$ in~\eqref{eq:ds},
	\begin{equation}\label{eq:dssquare}
		\nrm{\dS(z)}^2
		=
		\left\|\begin{bmatrix}\dA & \dB \end{bmatrix}\right\|_F^2
		+ 
		\left \| \begin{bmatrix} \dC & \dA_0 & \cdots & \dA_d
		\end{bmatrix} \right \|_F^2,
	\end{equation}
	and we can also write the constraint $\dS(\lambda)x= S(\lambda)x$ 
	on $\dS(\lambda)$ equivalently to 
	\begin{equation}\label{eq:temp1}
	\begin{bmatrix}\dA & \dB \end{bmatrix} x =
	\begin{bmatrix} (A-\lambda I) & B \end{bmatrix} x 
	\quad \text{and} \quad
	\begin{bmatrix}\dC & \dP(\lambda) \end{bmatrix} x
	=\begin{bmatrix} C & P(\lambda) \end{bmatrix}x.
	\end{equation} 
	Here, the second equation is further equivalent to 
	\begin{equation}\label{eq:allblo3}
		\begin{bmatrix}\dC & \Delta {A_0} & \cdots & \Delta {A_d}\end{bmatrix}
	\begin{bmatrix}x_1^T & x_2^T & \lambda x_2^T & \ldots &\lambda^d x_2^T \end{bmatrix}^T
	=\begin{bmatrix}C & P(\lambda) \end{bmatrix}x,
	\end{equation}
	and we have partitioned $x=\begin{bmatrix} x_1^T&x_2^T\end{bmatrix} ^T$ with  $x_1 \in
	\C^{r}$ and $x_2 \in \C^n$.

	We can minimize the two components of $\nrm{\dS(z)}^2$ in~\eqref{eq:dssquare} separately. 
	First, by~\Cref{lem:mapping},
	among all perturbation $\big[\dA ~ \dB \big]$
	satisfying the first equation of~\eqref{eq:temp1}, 
	\begin{equation}\label{eq:dadb0}
		\begin{bmatrix}\widehat\dA & \widehat\dB \end{bmatrix}
		=\begin{bmatrix} (A-\lambda I_r) & B\end{bmatrix} xx^{*}/\|x\|^2
	\end{equation}
	achieves the minimal Frobenius norm, which is given by
	\begin{equation}\label{eq:dadb}
		\left\|\begin{bmatrix}\widehat\dA & \widehat\dB \end{bmatrix}\right\|_F^2
		=
	\frac{\left\|\begin{bmatrix}(A-\lambda I_r) & B \end{bmatrix} x\right \|^2}{\|x\|^2}=\frac{x^*G_1x}{x^*x},
	\end{equation}
	where $G_1$ is given by~\eqref{eq:gmat}.
	Similarly, among all perturbation 
	$\begin{bmatrix} \dC & \dA_0 & \cdots & \dA_d \end{bmatrix}$
	satisfying~\eqref{eq:allblo3},
	the following matrix achieves the minimal Frobenius norm 
	\begin{equation}\label{eq:dcda0}
	\begin{bmatrix} \widehat\dC & \widehat \dA_0 & \cdots & \widehat \dA_d \end{bmatrix}
	= \begin{bmatrix} C & P(\lambda) \end{bmatrix} x\widetilde
	x^{*}/\|\widetilde x\|^2,
	\end{equation}
	where 
	$\widetilde x = \begin{bmatrix} x_1^T & x_2^T & \lambda x_2^T & \cdots & \lambda^d x_2^T\end{bmatrix}^T$,
	with the Frobenius norm satisfying
	\begin{equation}\label{eq:dcda}
		\left \| 
		\begin{bmatrix} 
			\widehat\dC & \widehat\dA_0 & \cdots & \widehat \dA_d 
		\end{bmatrix} 
		\right \|_F^2
		=
		\frac{\left \|\begin{bmatrix} C & P(\lambda) \end{bmatrix}x\right \|^2}
		{\left \|\widetilde x \right \|^2} 
		=\frac{x^*G_2x}{x^*(H_1+\gamma H_2)x},
	\end{equation}
	where $G_2,H_1$ and $H_2$ are as defined in~\eqref{eq:gmat}, and we used
	$\|\widetilde x\|^2= \|x_1\|^2+\gamma\|x_2\|^2$.

	Therefore, 
	by \eqref{eq:dssquare}, 
	the sum of \eqref{eq:dadb} and~\eqref{eq:dcda}
	gives the minimal perturbation norm $\nrm{\dS(z)}^2$
	for the inner optimization of~\eqref{eq:error2}.
	We thus prove~\eqref{eq:allbloexp}.
\end{proof}

From the proof of~\Cref{thm:full}, 
we can also see that the optimal perturbation $\dS_{\star}(z)$ 
achieving the backward error in~\eqref{eq:error}
can be generated by the perturbation matrices~\eqref{eq:dadb0} 
and~\eqref{eq:dcda0} with $x$ being a minimizer of the 
SRQ2 minimization~\eqref{eq:allbloexp}.

\subsection{Backward Error with Partial Perturbations}\label{sbackerror}
By partial perturbations,  we refer to the cases where only 
some particular blocks within $S(z)$ are perturbed.
Let $\Pi\subset\text{\{$A$, $B$, $C$, $P$\}}$
be a subset indicating the names of blocks to be perturbed. 
A structured perturbation of $S(z)$ corresponding to $\Pi$ 
is denoted by 
\begin{equation}\label{eq:dps}
	\dpS (z) :=
	\begin{bmatrix}
	\Delta A \cdot {\bf 1}_\Pi(A)  & \Delta B \cdot {\bf 1}_\Pi(B)  \\
	\Delta C \cdot {\bf 1}_\Pi(C)   & \Delta P(z) \cdot {\bf 1}_\Pi(P)   
	\end{bmatrix},
\end{equation}
where $\Delta A$, $\Delta B$, $\Delta C$, and $\Delta P(z)$ 
are as from $\dS(z)$ in~\eqref{eq:ds}, 
and ${\bf 1}_\Pi(x)$ is an indicator function,
i.e., ${\bf 1}_\Pi(x)=1$ if $x\in\Pi$, and 
${\bf 1}_\Pi(x)=0$ if $x\not\in\Pi$.
Similar to \Cref{def:error}, we define the 
backward error for $S(z)$ with a partial perturbation $\dpS (z)$
as follows.
\begin{definition}\label{def:perror}
	Let $\Pi\subset\{A, B, C, P\}$ indicate the names of the blocks to
	be perturbed.
	The backward error of an approximate eigenvalue $\lambda\in\C$ 
	of the Rosenbrock system matrix $S(z)$ as in~\eqref{eq:rosenbrock},
	subject to the perturbation from the set
	\begin{equation} \label{eq:bps}
		\ppset:=\{\dpS(z)\colon \text{$\dpS(z)$ is in the form of~\eqref{eq:dps}}\},
	\end{equation}
	is defined as 
	\begin{equation}\label{eq:perror}
		\eta(\lambda; \Pi ):=\min\left\{ \nrm{\dS(z)}
			: 
		\text{$\dS(z)\in\pset^\Pi$, $\det(S(\lambda)-\dS(\lambda))=0$}
		\right\},
	\end{equation}
	where we set $\eta(\lambda; \Pi )=\infty$, if
	the constrained optimization problem is {\em infeasible}, 
	i.e., there is no perturbation $\dS(z)$ within $\ppset$ 
	that satisfies $\det(S(\lambda)-\dS(\lambda))=0$.
\end{definition}
%

\begin{remark}\label{rmk:etap}
\rm
Due to the constraint on block structures of perturbation as in~\eqref{eq:dps},
the optimization problem~\eqref{eq:perror} may be {\em infeasible}.
In such a case, the backward error $\eta(\lambda; \Pi ) =\infty$ is unbounded. 
This is in contrast to the case of full perturbation~\eqref{eq:error}, 
where the backward error is always finite.
On the other hand, provided that the minimization~\eqref{eq:perror} is feasible,
the minimal perturbation $\dS_{\star}(z)\in\ppset$
in~\eqref{eq:perror} can always be attained,
as can be justified by using the same arguments as in~\Cref{rmk:eta}.
\end{remark}

Similar to~\eqref{eq:error2}, it is straightforward to write the optimization problem~\eqref{eq:perror} as 
\begin{align}
\eta(\lambda; \Pi)&=\min\left\{ \nrm{\dS(z)} : 
\text{$\dS(z)\in\ppset$, 
$x\in\bbS^{r+n}$,
$S(\lambda)x=\dS(\lambda)x$} \right\}.
\end{align}
We can then turn it into a 
bi-level optimization problem, by first fixing $x$ and minimizing over
$\nrm{\dS(z)}$, as given by 
\begin{equation} \label{eq:perror2}
\eta(\lambda; \Pi) = \min_{x\in\bbS^{r+n}} \
\varphi(x;\Pi),
\end{equation}
where $\varphi(x;\Pi)$ denotes the solution of the inner optimization
problem, i.e.,
\begin{equation} \label{eq:perror2inner}
\varphi(x;\Pi):=
\min \left\{ \nrm{\dS(z)} : 
\text{$\dS(z)\in\ppset$, $(S(\lambda)-\dS(\lambda))x=0$} \right\},
\end{equation}
and we assign $\varphi(x;\Pi)=\infty$ 
if the minimization problem is infeasible.

Like the case of full perturbation in~\eqref{eq:error},
we can show that 
the solution of the minimization problem~\eqref{eq:perror2inner} still 
admits an expression as the sum of (at most) two generalized Rayleigh
quotients in $x$.
However, the derivation of those solutions will 
depend on the particular block structure 
specified in $\Pi$.
For analyses, we will categorise the perturbations into three classes,
based on the number of blocks in $\Pi$. 
They are
(i) perturbing one single block from $A,B,C$, and $P(z)$ of $S(z)$,
(ii) perturbing any two blocks, 
and 
(iii) perturbing any three blocks.

\subsubsection{Perturbing a Single Block} 
Consider the backward error of $S(z)$ with only one of the blocks
from $A$, $B$, $C$, and $P(z)$ to be perturbed.
Specifically, we refer to $\eta(\lambda; \Pi)$ as defined
by~\eqref{eq:perror} with $\Pi=\{X\}$, for $X\in\{A,B,C,P\}$.
It will be sufficient to focus on the cases of 
$\eta(\lambda; A)$, $\eta(\lambda; B)$,
and $\eta(\lambda; P)$.
Due to symmetry, the backward error $\eta(\lambda; C)$ 
can be derived by the formula of
$\eta(\lambda; B)$ applied to the transposed system $S(z)^T$; see~\Cref{rmk:etac}.

We begin with a general-purpose lemma on partial perturbations
involving a single block before presenting the main result. 

\begin{lemma}\label{lem:1b}
	Consider the minimization problem
	\begin{eqnarray}\label{eq:optb1}
	\begin{aligned}
		\min \quad  \sum_{j=0}^d\|\Delta_{j}\|_F^2 
		\quad \text{s.t.} \quad 
			& \text{$\Delta_{j}\in \C^{r_1,n_1}$, for $j=0,1,\dots,d$,} \\
			& 
			\begin{bmatrix} 
				M_{11} - \sum_{j=0}^d \mu^j \Delta_{j} & M_{12}\\ 
				M_{21} & M_{22} 
			\end{bmatrix} 
	\begin{bmatrix} x_1 \\ x_2 \end{bmatrix} =0,
	\end{aligned}
	\end{eqnarray}
	where $\mu\in \C$,
	$M_{ij} \in \C^{r_i,n_j}$, and 
	$x_j\in\C^{n_j}$, for $i,j\in\{1,2\}$,
	are given.
	We assume $r_1+r_2=n_1+n_2$,
	$M:=\begin{bmatrix}M_{11} & M_{12} \\ M_{21} & M_{22} \end{bmatrix}$
	is non-singular,
	and 
	$x:=\begin{bmatrix}x_1 \\ x_2\end{bmatrix}\neq 0$.
	\begin{enumerate}[(a)]
	\item \label{lem:1b:a}
		The minimization problem~\eqref{eq:optb1} is feasible if and only if $M$ and $x$ satisfy
			\begin{equation}\label{eq:feasible}
			\text{$\begin{bmatrix}M_{21} &M_{22}\end{bmatrix}x=0$
			\quad and\quad $x_1\neq 0$}.
			\end{equation}

	\item \label{lem:1b:b}
		Under conditions~\eqref{eq:feasible}, 
		the minimizer $\{\Delta_{j\star}\}_{j=0}^d$ of~\eqref{eq:optb1}
		satisfies
		\begin{equation}\label{eq:optdlt}
			\sum_{j=0}^d\|\Delta_{j\star}\|_F^2 
			= \frac{1}{\sum_{j=0}^d |\mu|^{2j}}\cdot \frac{\|\begin{bmatrix} M_{11} & M_{12}
				\end{bmatrix} x\|^2}{ \|x_1\|^2}.
		\end{equation}

	\end{enumerate}
\end{lemma}

\begin{proof}
	For item~\eqref{lem:1b:a},
	we first reformulate the equality constraints in~\eqref{eq:optb1} as 
	\begin{equation}\label{eq:upper}
		\left\{
		\begin{aligned}
		\begin{bmatrix}
			\Delta_{0}& \Delta_{1}&\cdots & \Delta_{d}
		\end{bmatrix}
		\begin{bmatrix}x_1^T&\mu x_1^T&\cdots &\mu^d x_1^T\end{bmatrix}^T
		& =
		\begin{bmatrix}M_{11}&M_{12} \end{bmatrix} x,\\
		0 & = \begin{bmatrix}M_{21} &M_{22}\end{bmatrix}x.
		\end{aligned}\right.
	\end{equation}
	For feasibility, 
	we will need $[M_{21},M_{22}]x = 0$.
	This condition also implies that $[M_{11},M_{12}]x\neq 0$, because 
	$Mx\neq 0$ (recall that $M$ is non-singular and $x\neq 0$).
	Since the first equation in~\eqref{eq:upper} has a non-zero right
	hand side, by~\Cref{lem:mapping},
	there exist solution $\{\Delta_j\}$ if and only if $x_1\neq 0$.
	This proves the conditions in~\eqref{eq:bs}. 

	Item~\eqref{lem:1b:b} follows directly from the first equation
	of~\eqref{eq:upper} and~\Cref{lem:mapping}.
\end{proof}

\begin{theorem}\label{thm:1b}
Let $S(z)$ be a Rosenbrock system matrix given by~\eqref{eq:rosenbrock},
$\lambda\in\C$ be a scalar with $\det(S(\lambda))\neq 0$,
$G_1,G_2,H_1,H_2$ be Hermitian matrices from~\eqref{eq:gmat},
and $U\in \C^{r+n,p}$ and $V\in \C^{r+n,k}$ 
form orthonormal basis of the null spaces
of $G_1$ and $G_2$, respectively.
Then, it holds that 
$V^*G_1V\succ 0$ and $U^*G_2U\succ 0$,
and we have the following characterizations 
for the eigenvalue backward errors. 
\begin{enumerate}[(a)]
	\item \label{thm:1b:a}
		$\eta(\lambda;A)=\infty$ if and only if $V^*H_1V=0$;
		Otherwise, 
	\begin{equation}\label{eq:perA}
	\left(\eta(\lambda;A)\right)^2
	= 
	\min_{y \in \bbS^{k}}
	\left(\frac{y^*V^*G_1Vy}{y^*V^*H_1Vy}\right) 
	\equiv \lambda_{\min}\left(V^*G_1V,\, V^*H_1V\right).
	\end{equation}

	\item \label{thm:1b:b}
		$\eta(\lambda;B)=\infty$ if and only if $V^*H_2V=0$;
		Otherwise, 
		\begin{equation}\label{eq:perB}
		\left(\eta(\lambda;B)\right)^2
		=
		\min_{y \in \bbS^{k}}
		\left(\frac{y^*V^*G_1Vy}{y^*V^*H_2Vy}\right)
		\equiv
		\lambda_{\min}(V^*G_1V,\, V^*H_2V).
	\end{equation}

	\item \label{thm:1b:p}
		$\eta(\lambda;P)=\infty$ if and only if $ U^*H_2U = 0$; 
		Otherwise, with $\gamma\equiv \sum_{j=0}^d |\lambda|^{2j}$, 
		\begin{equation}\label{eq:perP}
		\left(\eta(\lambda;P)\right)^2
		= 
		\frac{1}{\gamma} \cdot
		\min_{y \in \bbS^{p}}
		\left(\frac{y^*U^*G_2Uy}{y^*U^*H_2Uy}\right)
		\equiv 
		\frac{1}{\gamma} \cdot \lambda_{\min}(U^*G_2U,\, U^*H_2U).
	\end{equation}
\end{enumerate}
\end{theorem}

\begin{proof}
	To begin with, we have 
	$G_1+G_2\equiv S(\lambda)^*S(\lambda)\succ 0$,
	due to $\det(S(\lambda))\neq 0$.
	Then a quick verification shows
	$V^*G_1V\succ 0$ and $U^*G_2U\succ 0$.

	For item~\eqref{thm:1b:a}, it follows from~\eqref{eq:perror2} with
	$\Pi=\{A\}$ that 
	\begin{equation}\label{eq:etaa}
		\eta(\lambda;A)=\min_{x\in\bbS^{r+n}}
		\left( \min\left\{ \|\dA\|_F \colon 
		\begin{bmatrix} (A-\lambda I_r)-\dA & B \\ C &
		P(\lambda)\end{bmatrix} x = 0
	\right\}\right).
	\end{equation}
	By~\Cref{lem:1b}, with $\mu=0$ and $M=S(\lambda)$, 
	the inner optimization above has solution 
	\begin{equation}\label{eq:das2}
		\|\dA_{\star}\|_F^2 = \frac{\|\begin{bmatrix} A-\lambda I_r & B
		\end{bmatrix} x\|^2}{\|x(1:r)\|^2} 
		= \frac{x^*{G_1}x}{x^*{H_1}x},
	\end{equation}
	where $G_1$ and $H_1$ are as defined in~\eqref{eq:gmat},
	and we assumed $x$ satisfies 
	\begin{equation}\label{eq:feas1ba}
		x(1:r)\neq 0 \quad \text{and} \quad
		\begin{bmatrix} C & P(\lambda) \end{bmatrix}x= 0,
	\end{equation}
	as according to the feasibility condition~\eqref{eq:ds}
	(otherwise $\|\dA_{\star}\|_F^2=\infty$).
	The second constraint, $[C, P(\lambda)]x=0$, implies we can parameterize $x$ as 
	\begin{equation}\label{eq:parx}
		x = Vy \quad\text{for some $y\in\bbS^k$},
	\end{equation}
	where $V$ is a basis matrix of the null space of
	$[ C, P(\lambda)]$ (and of $G_2$~\eqref{eq:gmat} as well).
	Consequently, plugging~\eqref{eq:parx} into~\eqref{eq:das2},
	and further into~\eqref{eq:etaa}, we obtain
	\[
		\left(\eta(\lambda; A)\right)^2 =
		\min
		\left\{\frac{y^*V^*{G_1}Vy}{y^*V^*{H_1}Vy}
		\; :\; y \in \bbS^k,\, y^*V^*{H_1}Vy \neq 0 \right\},
	\]
	where $y^*V^*{H_1}Vy \neq 0$ 
	is due to $x^*H_1x\equiv \|x(1:r)\|^2 \neq 0$
	from~\eqref{eq:feas1ba}.
	Clearly, this optimization is feasible if and only if
	$V^*{H_1}V\neq 0$. 
	We thus proved the first equality in~\eqref{eq:perA},
	upon noticing that the Rayleigh quotient 
	yields $+\infty$ for $y$ satisfying $y^*(V^*{H_1}V)y =0$,
	because of $V^*{G_1}V \succ 0$. 
	The second equality is by the well-known 
	minimization principle of 
	Hermitian eigenvalue problems; 
	see, e.g.,~\cite[Chap.~VI]{Stewart:1990}.

	For item~\eqref{thm:1b:b}, it follows from~\eqref{eq:perror2} with
	$\Pi=\{B\}$ that 
	\begin{align*}
		\eta(\lambda;B)
		&=\min_{x\in\bbS^{r+n}}
		\left( \min\left\{ \|\dB\|_F \colon 
			\begin{bmatrix} (A-\lambda I_r)  & B +\dB\\ C & P(\lambda) \end{bmatrix}
			\begin{bmatrix}x_1 \\ x_2 \end{bmatrix}  = 0
		\right\}\right)\\
		&=\min_{x\in\bbS^{r+n}}
		\left( \min\left\{ \|\dB\|_F \colon 
				\begin{bmatrix}  B +\dB & (A-\lambda I_r) \\ P(\lambda)
				& C\end{bmatrix} \begin{bmatrix}x_2 \\ x_1 \end{bmatrix} = 0
		\right\}\right),
	\end{align*}
	where in the second equation, we have moved $\dB$ to the $(1,1)$-block by switching the 
	column blocks in the linear constraints.
	The rest of the proof is analogous to
	item~\eqref{thm:1b:a}.

	For item~\eqref{thm:1b:p}, it follows from~\eqref{eq:perror2} with $\Pi=\{P\}$ that
	\begin{align*}
		\eta(\lambda;P)
		&=\min_{x\in\bbS^{r+n}}
		\left(\!\! \min\left\{ \sqrt{\sum_{j=0}^d\|\dA_j\|_F^2} \colon \!
			\begin{bmatrix} A-\lambda I_r  & B \\ C & P(\lambda) +
			\sum\limits_{j=0}^d \lambda^j\dA_j\end{bmatrix} 
			\!\!\begin{bmatrix}x_1 \\ x_2 \end{bmatrix}\! =\! 0
		\right\}\!\!\right)\\
		&=\min_{x\in\bbS^{r+n}}
		\left(\!\! \min\left\{ \sqrt{\sum_{j=0}^d\|\dA_j\|_F^2} \colon \!
			\begin{bmatrix} P(\lambda) + \sum\limits_{j=0}^d \lambda^j \dA_j& C\\ 
				B & A-\lambda I_r \end{bmatrix} 
			\!\!\begin{bmatrix}x_2 \\ x_1 \end{bmatrix}\! =\! 0
		\right\}\!\!\right),
	\end{align*}
	where in the second equation we have switched the row and column
	blocks in the linear constraint. 
	The rest of the proof follows item~\eqref{thm:1b:a},
	using~\Cref{lem:1b} with $\mu=\lambda$.
\end{proof}

\begin{remark}\label{rmk:etac}
\rm
To obtain the backward error $\eta(\lambda; C)$ of the Rosenbrock
system matrix $S(z)$, we can consider the transposed system
matrix $(S(z))^T$,
as denoted by
\begin{equation}\label{eq:hats}
	\widehat S(z):=
	\begin{bmatrix} \widehat A-z I_r & \widehat B \\ \widehat C & \sum_{j=0}^d
	z^j\widehat A_j\end{bmatrix}
	\equiv 
	\begin{bmatrix} A^T-z I_r & C^T \\ B^T & \sum_{j=0}^d
	z^jA_j^T\end{bmatrix}.
\end{equation}
Then perturbing the $C$ block of $S(z)$ by $\dC$ is 
the same as perturbing the $\widehat B$ block of 
$\widehat S(z)$ by $\dC^T$.
Since ${\|\dC^T\|}_F\equiv {\|\dC\|}_F$,
we have $\eta(\lambda;C) =\widehat\eta(\lambda; \widehat B)$, 
where $\widehat\eta(\lambda; \widehat B)$ is the eigenvalue backward error 
of the new system $\widehat S(z)$ subject to a perturbation to 
its $\widehat B$ block.
Here, we can evaluate $\widehat\eta(\lambda; \widehat B)$
through~\Cref{thm:1b}~\eqref{thm:1b:b} applied to $\widehat S(z)$.
\end{remark}

\subsubsection{Perturbing Any Two Blocks}
We now consider the cases where exactly two blocks of $S(z)$
are to be perturbed, 
i.e., $\eta(\lambda;\Pi)$ with $\Pi=\{X_1,X_2\} \subset \{A,B,C,P\}$.
We have six such cases in total,
and it is sufficient to focus on the four cases
with $\Pi=\{A,B\}$, $\{A,P\}$, $\{B,C\}$, and $\{C,P\}$.
The other two, 
with $\Pi=\{A,C\}$ and $\{B,P\}$, 
can be addressed using symmetry 
by considering the transposed system $(S(z))^T$,
as similar to~\Cref{rmk:etac}.
Again, we start by discussing a general-purpose lemma 
before establishing the formulas for the backward errors.

\begin{lemma}\label{lem:2b}
	Let $\mu\in \C$,
	$M_{ij} \in \C^{r_i,n_j}$, and 
	$x_j\in\C^{n_j}$, for $i,j\in\{1,2\}$, where $r_1+r_2=n_1+n_2$.
	Suppose that $x=[x_1^T,x_2^T]^T\neq 0$.

	\begin{enumerate}[(a)]
		
	\item \label{lem:2b:a}
	The minimization problem 
	\begin{equation}\label{eq:optb2a}
	\begin{aligned}
		\min \quad & \|\Delta_{11}\|_F^2+\sum_{j=0}^d\|\Delta_{j}\|_F^2\\
	\text{s.t.} \quad 
			   & \text{$\Delta_{11}\in \C^{r_1,n_1}$, $\Delta_{j}\in \C^{r_1,n_2}$, for $j=0,1,\dots,d$,} \\
				&
		\begin{bmatrix} 
	M_{11} - \Delta_{11} & M_{12}- \sum_{j=0}^d \mu^j \Delta_{j}\\
	M_{21} & M_{22} 
	\end{bmatrix} 
	\begin{bmatrix} x_1 \\ x_2 \end{bmatrix} =0
	\end{aligned}
	\end{equation}
	is feasible if and only if 
	$\begin{bmatrix} M_{21} & M_{22} \end{bmatrix} x=0$,
	where the minimal objective value 
	\begin{equation}\label{eq:optdlt2}
		\|\Delta_{11\star}\|_F^2 +
		\sum_{j=0}^d\|\Delta_{j\star}\|_F^2 
		= \frac{\|\begin{bmatrix} M_{11} & M_{12} \end{bmatrix}x \|^2}
		{\|x_1\|^2+(\sum_{j=0}^d |\mu|^{2j})\cdot \|x_2\|^2}.
	\end{equation}

	\item \label{lem:2b:b}
	The minimization problem 
	\begin{equation}\label{eq:optb2b}
	\begin{aligned}
		\min \quad & \|\Delta_{11}\|_F^2+\sum_{j=0}^d\|\Delta_{j}\|_F^2\\
		\text{s.t.} \quad 
			   	& \text{$\Delta_{11}\in \C^{r_1,n_1}$, $\Delta_{j}\in \C^{r_2,n_2}$, for $j=0,1,\dots,d$,} \\
				& \begin{bmatrix} 
					M_{11} - \Delta_{11} & M_{12}\\ 
					M_{21} & M_{22} - \sum_{j=0}^d \mu^j \Delta_{j}\\ 
				\end{bmatrix} 
				\begin{bmatrix} x_1 \\ x_2 \end{bmatrix} =0 
	\end{aligned}
	\end{equation} 
	has the minimal objective value given by 
	\begin{equation}\label{eq:optdlt2b}
		\|\Delta_{11\star}\|_F^2 +
	\sum_{j=0}^d\|\Delta_{j\star}\|_F^2
	= \frac{\|\begin{bmatrix} M_{11} & M_{12} \end{bmatrix} x\|^2}{\|x_1\|^2} 
		+\frac{\|\begin{bmatrix} M_{21} & M_{22} \end{bmatrix}
		x\|^2}{
		(\sum_{j=0}^d |\mu|^{2j})\cdot\|x_2\|^2},
	\end{equation}
	where we assign a value of $0$ to a fraction in the form of $\{\frac{0}{0}\}$.
	Here, the objective value~\eqref{eq:optdlt2b} is $\infty$, 
	i.e., the optimization~\eqref{eq:optb2b} is infeasible,
	if and only if either (i)~$x_1=0$ and $M_{12}x_2\neq 0$, or (ii)
	$x_2=0$ and $M_{21}x_1\neq 0$.
\end{enumerate}

\end{lemma}

\begin{proof}
	For item~\eqref{lem:2b:a},
	we can reformulate the constraints for $\Delta$'s in~\eqref{eq:optb2a} as 
	\begin{equation}\label{eq:upperb2a}
		\left\{
		\begin{aligned}
		\begin{bmatrix}
			\Delta_{11}\!\!&\Delta_{0}& \Delta_{1}&\cdots & \Delta_{d}
		\end{bmatrix}
		\begin{bmatrix}x_1^T\!\!&x_2^T&\mu x_2^T&\cdots &\mu^d x_2^T\end{bmatrix}^T
		& =
		\begin{bmatrix}M_{11}&M_{12} \end{bmatrix} x,\\
		0 & = \begin{bmatrix}M_{21} &M_{22}\end{bmatrix}x.
		\end{aligned}\right.
	\end{equation}
	In the first equation,
	because of $\begin{bmatrix}x_1^T&x_2^T&\cdots &\mu^d x_2^T\end{bmatrix}\neq 0$
	(due to $x=[x_1^T,x_2^T]\neq 0$),
	there always exists a solution 
	$[\Delta_{11},\Delta_{0}, \dots, \Delta_{d}]$
	according to \Cref{lem:mapping}.
	Moreover, the solution 
	$[\Delta_{11\star},\Delta_{0\star}, \dots, \Delta_{d\star}]$
	with the minimal Frobenius norm satisfies 
	\[ 
		\|[\Delta_{11\star},\Delta_{0\star},\dots,\Delta_{d\star}]\|_F=
	\|\begin{bmatrix}M_{11}&M_{12} \end{bmatrix} x\|/\|
	\begin{bmatrix}x_1^T&x_2^T&\mu x_2^T&\cdots &\mu^d
	x_2^T\end{bmatrix}\|.
	\]
	This leads to the optimal objective value~\eqref{eq:optdlt2}
	for the optimization~\eqref{eq:optb2a}, 
	provided that 
	$\begin{bmatrix}M_{21} &M_{22}\end{bmatrix}x=0$, 
	which ensures the second constraint in~\eqref{eq:upperb2a}
	is also feasible.

	For item~\eqref{lem:2b:b},
	we can reformulate the constraints for $\Delta$'s
	in~\eqref{eq:optb2b} as 
	\begin{equation}\label{eq:upperb2b}
		\left\{
		\begin{aligned}
		\Delta_{11} x_1 & =
		\begin{bmatrix}M_{11}&M_{12} \end{bmatrix} x,\\
		\begin{bmatrix}
			\Delta_{0}& \Delta_{1}&\cdots & \Delta_{d}
		\end{bmatrix}
		\begin{bmatrix}x_2^T&\mu x_2^T&\cdots &\mu^d x_2^T\end{bmatrix}^T
		& = \begin{bmatrix}M_{21} &M_{22}\end{bmatrix}x.
		\end{aligned}\right.
	\end{equation}
	We can consider the optimal $\Delta_{11\star}$ and $[\Delta_{0\star}, \dots, \Delta_{d\star}]$
	separately.

	For the first equation in~\eqref{eq:upperb2b},
	we derive from~\Cref{lem:mapping} that:
	(i)
	If $x_1\neq 0$, then the solution $\Delta_{11}$ with the minimal Frobenius 
	norm satisfies 
	\begin{equation}\label{eq:2b:x1}
		\|\Delta_{11\star}\|_F=
	\|\begin{bmatrix}M_{11}&M_{12} \end{bmatrix} x\|/ \|x_1\|;
	\end{equation}
	(ii)
	If $x_1=0$ and $\begin{bmatrix}M_{11}&M_{12} \end{bmatrix} x = M_{12}x_2=0$,
	then the solution with the minimal Frobenius norm is $\Delta_{11\star}=0$, 
	which also satisfies~\eqref{eq:2b:x1} by letting 
	the fraction $\frac{0}{0}$ to have value $0$; and
	(iii)
	If $x_1=0$ and $M_{12}x_2\neq 0$, then there is no solution
	$\Delta_{11}$ and the optimization~\eqref{eq:optb2b} is infeasible.

	For the second equation in~\eqref{eq:upperb2b},
	we derive from~\Cref{lem:mapping} that:
	(i)
	If $x_2\neq 0$, then the solution $[\Delta_{0}, \dots, \Delta_{d}]$
	with a minimal Frobenius norm satisfies 
	\begin{equation}\label{eq:2b:x2}
		\|[\Delta_{0\star},\dots,\Delta_{d\star}]\|_F=
	\|\begin{bmatrix}M_{21}&M_{22} \end{bmatrix} x\|/
	\| \begin{bmatrix}x_2^T&\mu x_2^T&\cdots &\mu^d
	x_2^T\end{bmatrix}\|;
	\end{equation}
	(ii)
	If $x_2=0$ and $\begin{bmatrix}M_{21}&M_{22} \end{bmatrix} x = M_{21}x_1=0$,
	then the solution with the minimal Frobenius norm
	is $[\Delta_{0\star},\dots,\Delta_{d\star}]=0$,
	which satisfies~\eqref{eq:2b:x1} by letting 
	the fraction $\frac{0}{0}$ to have value $0$; and
	(iii)
	If $x_2=0$ and $M_{21}x_1\neq 0$, then there is no solution
	$[\Delta_{0}, \dots, \Delta_{d}]$ and the optimization~\eqref{eq:optb2b} is infeasible.

	Combining~\eqref{eq:2b:x1} and~\eqref{eq:2b:x2}, we proved~\eqref{eq:optdlt2b}.
\end{proof}

\begin{theorem}\label{thm:2b}
Let $S(z)$ be a Rosenbrock system given by~\eqref{eq:rosenbrock},
$\lambda\in\C$ be a scalar with $\det(S(\lambda))\neq 0$,
$G_1,G_2,H_1,H_2$ be Hermitian matrices from~\eqref{eq:gmat},
and $U\in \C^{r+n,p}$ and $V\in \C^{r+n,k}$ 
form orthonormal basis of the null spaces
of $G_1$ and $G_2$, respectively.
\begin{enumerate}[(a)]
	\item \label{thm:2b:iab}
	$\eta(\lambda;A,B) <\infty$ and it satisfies 
	\begin{equation}\label{eq:AB}
		\left(\eta(\lambda;A,B)\right)^2=
		\min_{y\in\bbS^{k}}
		\left(\frac{y^*V^*G_1Vy}{y^*y} \right)
		\equiv
		\lambda_{\min}\left(V^* G_1 V\right).
	\end{equation}

	\item \label{thm:2b:iap} 
	$\eta(\lambda;A,P) <\infty$ and it satisfies 
	\begin{eqnarray}\label{eq:AP}
	(\eta(\lambda;A,P))^2= 
		\min_{x\in\bbS^{r+n}}
		\left(\frac{x^* G_1 x}{x^* H_1 x} +
		\frac{1}{\sum_{j=0}^d |\lambda|^{2j}} 
		\cdot \frac{x^* G_2 x}{x^* H_2 x}\right).
	\end{eqnarray}

	\item \label{thm:2b:ibc} 
	$\eta(\lambda;B,C) <\infty$ and it satisfies 
	\begin{eqnarray}\label{eq:BC}
		(\eta(\lambda;B,C))^2=
		\min_{x\in\bbS^{r+n}}
		\left(\frac{x^* G_1 x}{x^* H_2 x}+ \frac{x^* G_2 x}{x^*H_1 x} \right).
	\end{eqnarray}

	\item \label{thm:2b:icp} 
	$\eta(\lambda;C,P) <\infty$ and it satisfies 
	\begin{equation}\label{eq:CP2}
		\left(\eta(\lambda;C,P)\right)^2=
		\min_{y\in\bbS^{p}}
		\left(\frac{y^*U^* G_2 Uy}{y^*U^*H_{\lambda} Uy} \right)
		\equiv
		\lambda_{\min}\left(U^* G_2 U,\, U^*H_{\lambda} U \right),
	\end{equation}
	where $H_{\lambda}:=H_1+(\sum_{j=0}^d |\lambda|^{2j})\cdot H_2$,
	and we recall~\Cref{thm:2b} that $U^* G_2 U\succ 0$.
\end{enumerate}

\end{theorem}

\begin{proof}
	For item~\eqref{thm:2b:iab}, it follows from~\eqref{eq:perror2inner} 
	with $\Pi=\{A,B\}$ that 
	\begin{equation*}
		(\varphi(x;A,B))^2 = \!\!\!\!\! \min_{\substack{ \dA \in  \C^{r,r}\\ \dB \in \C^{r,n}}}
				\Bigg\{{\|\dA\|}_F^2+{\|\dB\|}_F^2 \colon
		\begin{bmatrix} (A-\lambda I_r)-\dA & B-\dB \\ C & P(\lambda)\end{bmatrix}x=0 \Bigg\}.
	\end{equation*}
	According to~\Cref{lem:2b}~\eqref{lem:2b:a} with $M=S(\lambda)$ and $\mu=0$, 
	the optimal objective value $\varphi(x;A,B)<\infty$ if and only if $\begin{bmatrix}C &
	P(\lambda) \end{bmatrix} x = 0$ (i.e., $G_2x=0$), where 
	\[
		(\varphi(x;A,B))^2 = \frac{\|\begin{bmatrix}A-\lambda I_r & B \end{bmatrix} x\|^2}{\|x\|^2}
		= \frac{x^*G_1x}{x^*x}.
	\]
	By~\eqref{eq:perror2} and the 
	parameterization $x=Vy$ of the null vector $x$ of $G_2$,
	we obtain~\eqref{eq:AB}.

	For item~\eqref{thm:2b:iap}, it follows from~\eqref{eq:perror2inner} with $\Pi=\{A,P\}$ that 
	\begin{equation*}
		(\varphi(x;A,P))^2 = \!\!\!\!\! \min_{\substack{ \dA \in
		\C^{r,r}\\ \Delta_j \in \C^{n,n} \\ j=0,\dots,d}}
		\!\!\left\{\!
		\|\dA\|_F^2+\sum_{j=0}^d \|\Delta_j\|_F^2 
		\colon
		\!\!\left[
			\begin{smallmatrix} 
				(A-\lambda I_r)-\dA & B \\
				C & P(\lambda)-\sum\limits_{j=0}^d \lambda^j\Delta_j 
		\end{smallmatrix}
		\right] x=0 
		\!\right\}.
	\end{equation*}
	According to~\Cref{lem:2b}~\eqref{lem:2b:b} with $M=S(\lambda)$ and $\mu=\lambda$, 
	the optimal objective value 
	\[
		(\varphi(x;A,P))^2 = 
		\frac{\|\begin{bmatrix} (A-\lambda I_r) & B \end{bmatrix} x\|^2}{\|x_1\|^2} 
			+\frac{\|\begin{bmatrix} C  & P(\lambda) \end{bmatrix}x\|^2}
			{(\sum_{j=0}^d|\lambda|^{2j})\cdot\|x_2\|^2},
	\]
	which leads to~\eqref{eq:AP} by recalling~\eqref{eq:perror2}. 
	By taking a particular $x$ satisfying $x_1\neq 0$ and $x_2\neq 0$,
	we can derive from~\eqref{eq:AP} that 
	$\eta(\lambda;A,P)\leq \varphi(x;A,P)<\infty$.

	For item~\eqref{thm:2b:ibc}, it follows from~\eqref{eq:perror2inner} 
	with $\Pi=\{B,C\}$ that 
	\begin{equation*}
		(\varphi(x;B,C))^2 = \!\!\!\!\! \min_{\substack{ \dB \in  \C^{r,n}\\ \dC \in \C^{n,r}}}
				\Bigg\{{\|\dB\|}_F^2+{\|\dC\|}_F^2 \colon
		\begin{bmatrix} A-\lambda I_r & B-\dB \\ C-\dC & P(\lambda)
	\end{bmatrix}x=0 \Bigg\}.
	\end{equation*}
	We can switch the row blocks in the constraint $(S(\lambda)-\dS(\lambda))x=0$
	to place $\dC$ and $\dB$ in the diagonal blocks.
	The rest of the proof is completely analogous to item~\eqref{thm:2b:iap}.

	Finally, for item~\eqref{thm:2b:icp}, it follows from~\eqref{eq:perror2inner} 
	with $\Pi=\{C,P\}$ that 
	\[
		(\varphi(x;C,P))^2 = \!\!\!\!\! \min_{\substack{ \dC \in
		\C^{n,r}\\ \Delta_j \in \C^{n,n} \\ j=0,\dots,d}}
		\!\!\left\{\!
		\|\dC\|_F^2+\sum_{j=0}^d \|\Delta_j\|_F^2 
		\colon
		\!\!\left[
			\begin{smallmatrix} 
				A-\lambda I_r & B \\
				C-\dC & P(\lambda)-\sum\limits_{j=0}^d \lambda^j\Delta_j 
		\end{smallmatrix}
		\right] x=0 
		\!\right\}.
	\]
	By switching the row blocks in the constraint
	$(S(\lambda)-\dS(\lambda))x=0$, we can place the perturbations
	$\Delta$'s to the first row blocks, and then the proof
	follows~item~\eqref{thm:2b:iab}.
\end{proof}

\begin{remark}\label{rmk:etacbp}
\rm
To obtain the backward error $\eta(\lambda; A, C)$ 
and $\eta(\lambda; B, P)$ of the Rosenbrock system matrix $S(z)$,
we can again consider the transposed
Rosenbrock system matrix $(S(z))^T$ as given by~\eqref{eq:hats}.
By block symmetry in transposing, we have 
\begin{equation}\label{eq:acbp}
\eta(\lambda;A,C) \equiv \widehat\eta(\lambda; \widehat A,\widehat B)
\quad\text{and}\quad
\eta(\lambda;B,P) \equiv \widehat\eta(\lambda; \widehat C,\widehat P),
\end{equation}
where $\widehat\eta(\lambda; \Pi)$ denotes
the eigenvalue backward error of the new system matrix $\widehat S(z)$ subject to a 
perturbation pattern $\Pi$.
Now, for the perturbation patterns
$\Pi=\{\widehat A,\widehat B\}$ and $\{\widehat C,\widehat P\}$
under consideration, their corresponding backward error can be obtained
quickly by the formulas~\eqref{eq:AB} and~\eqref{eq:CP2} applied to the
transposed system matrix.
\end{remark}

\subsubsection{Perturbing Any Three Blocks}
Finally, we consider the cases where exactly three blocks of $S(z)$ 
are to be perturbed. 
There are four such type of perturbations,
and we focus on the three cases with 
$\Pi=\{A,B,C\}$, $\{A,B,P\}$, and $\{A,C,P\}$.
The other case with $\Pi=\{B,C,P\}$ is readily  
solved by the transposed system matrix.


\begin{lemma}\label{lem:3b}
Let $\mu_i\in \C$,
$M_{ij} \in \C^{r_i,n_j}$, and 
$x_j\in\C^{n_j}$, for $i,j\in\{1,2\}$, where $r_1+r_2=n_1+n_2$.
Suppose that $x=[x_1^T,x_2^T]^T\neq 0$.
The minimization problem
\begin{equation}\label{eq:optb3}
\begin{aligned}
	\min \quad & \|\Delta_{12}\|_F^2+\sum_{j=0}^{d_1}\|\Delta_{j}\|_F^2+ 
	\sum_{j=0}^{d_2}\|\widetilde \Delta_{j}\|_F^2\\
\text{s.t.} \quad 
		   & \text{$\Delta_{12}\in \C^{r_1,n_2}$, $\Delta_{j}\in \C^{r_1,n_1}$, for $j=0,1,\dots,d_1$,} \\
		   & \text{$\widetilde \Delta_{j}\in \C^{r_1,n_1}$, for $j=0,1,\dots,d_2$,} \\
		  &
\begin{bmatrix} 
M_{11} - \sum_{j=0}^{d_1} \mu_1^j \Delta_{j} & M_{12} - \Delta_{12}\\
M_{21} - \sum_{j=0}^{d_2} \mu_2^j \widetilde \Delta_{j}  & M_{22} 
\end{bmatrix} 
\begin{bmatrix} x_1 \\ x_2 \end{bmatrix} =0
\end{aligned}
\end{equation}
has an optimal objective value given by 
\begin{equation}\label{eq:optdlt3b}
\|\Delta_{12\star}\|_F^2 
+ \sum_{j=0}^{d_1}\|\Delta_{j\star}\|_F^2 
+ \sum_{j=0}^{d_2}\|\widetilde \Delta_{j\star}\|_F^2 
=
\frac{\|[\begin{smallmatrix} M_{11} & M_{12}
	\end{smallmatrix}] x\|^2}{c_1\|x_1\|^2+\|x_2\|^2} 
	+\frac{\|[\begin{smallmatrix} M_{21} & M_{22} \end{smallmatrix}]
	x\|^2}{c_2\cdot\|x_1\|^2},
\end{equation}
where $c_i=\sum_{j=0}^{d_i} |\mu_i|^{2j}$ for $i=1,2$, and 
we assign a value of $0$ to a fraction in the form of $\{\frac{0}{0}\}$.
The optimal objective value~\eqref{eq:optdlt3b} is $\infty$,
if and only if $x_1=0$ and $M_{22}x_2\neq 0$, 
indicating the optimization~\eqref{eq:optb3} is infeasible.


\end{lemma}

\begin{proof}
	We write the constraints for $\Delta$'s in~\eqref{eq:optb3} equivalently to 
	\begin{equation}\label{eq:upperb3}
		\left\{
		\begin{aligned}
		\begin{bmatrix}
			\Delta_{0}& \Delta_{1}&\!\!\!\cdots\!\!\! & \Delta_{d_1} & \Delta_{12}
		\end{bmatrix}
		\begin{bmatrix}x_1^T&\mu_1 x_1^T&\!\!\!\cdots\!\!\! &\mu_1^{d_1} x_1^T& x_2^T\end{bmatrix}^T
		&\!\! =
		\begin{bmatrix}M_{11}&M_{12} \end{bmatrix} x,\\
		\begin{bmatrix}
			\widetilde \Delta_{0}& \widetilde \Delta_{1}&\!\!\!\cdots\!\!\! &
			\widetilde \Delta_{d_2}
		\end{bmatrix}
		\begin{bmatrix}x_1^T&\mu_2 x_1^T&\!\!\!\cdots\!\!\! &\mu_2^{d_2} x_1^T\end{bmatrix}^T
		&\!\! = \begin{bmatrix}M_{21} &M_{22}\end{bmatrix}x.
		\end{aligned}\right.
	\end{equation}
	In the first equation,
	because of $\begin{bmatrix}x_1^T&\dots&\mu_1^{d_1}x_1^T&x_2^T\end{bmatrix}\neq 0$
	(due to $x=[x_1^T,x_2^T]\neq 0$),
	by \Cref{lem:mapping}
	there always exists solution 
	$[\Delta_{0}, \dots, \Delta_{d_1},\Delta_{12}]$
	and the solution with the minimal Frobenius norm satisfies 
	\begin{equation}\label{eq:3b:up}
		\|[\Delta_{0\star}, \dots, \Delta_{d_1\star},\Delta_{12\star}]\|_F^2
		= 
		\frac{\|\begin{bmatrix}M_{11}&M_{12} \end{bmatrix} x\|^2}
		{\| \begin{bmatrix}x_1^T&\mu_1x_1^T&\cdots &\mu_1^{d_1}x_1^T&
			x_2^T\end{bmatrix}\|^2}.
	\end{equation}
	In the second equation of~\eqref{eq:upperb3}, 
	we can derive from~\Cref{lem:mapping} that:
	(i) If $x_1\neq 0$, then
	the solution $[\widetilde \Delta_{0}, \dots, \widetilde \Delta_{d_2}]$
	with the minimal Frobenius norm satisfies 
	\begin{equation}\label{eq:3b:down}
		\|[\widetilde \Delta_{0\star}, \dots, \widetilde \Delta_{d_2\star}]\|_F^2
		= 
		\frac{\|\begin{bmatrix}M_{21}&M_{22} \end{bmatrix} x\|^2}
		{\| \begin{bmatrix}x_1^T&\mu_2x_1^T&\cdots
		&\mu_2^{d_2}x_1^T\end{bmatrix}\|^2};
	\end{equation}
	(ii) If $x_1=0$ and $[M_{21},M_{22}]x=M_{22}x_2=0$,
	then the solution with the minimal Frobenius norm is
	$[\widetilde \Delta_{0\star}, \dots, \widetilde \Delta_{d_2\star}]=0$, 
	which also satisfies~\eqref{eq:3b:down} by letting the 
	fraction $\frac{0}{0}$ to have a value $0$; and
	(iii) If $x_1=0$ and $M_{22}x_2\neq 0$,
	then there is no solution $[\widetilde \Delta_{0}, \dots, \widetilde \Delta_{d_2}]$
	and the optimization~\eqref{eq:optb3} is infeasible.

	Combining~\eqref{eq:3b:up} and~\eqref{eq:3b:down}, we obtain~\eqref{eq:optdlt3b}.
\end{proof}

\begin{theorem}\label{thm:3b}
Let $S(z)$ be a Rosenbrock system matrix as given by~\eqref{eq:rosenbrock},
$\lambda\in\C$ be a scalar with $\det(S(\lambda))\neq 0$,
and $G_1,G_2,H_1,H_2$ be Hermitian matrices from~\eqref{eq:gmat}.

\begin{enumerate}[(a)]
	\item \label{thm:3b:iabc}
	$\eta(\lambda;A,B,C) <\infty$ and it satisfies 
	\begin{equation}\label{eq:abc}
		(\eta(\lambda;A,B,C))^2=
		\min_{x\in\bbS^{r+n}}
		\left(\frac{x^* G_1 x}{x^*x} + \frac{x^* G_2 x}{x^*H_1x}\right).
	\end{equation}

	\item \label{thm:3b:iabp}
	$\eta(\lambda;A,B,P) <\infty$ and it satisfies 
	\begin{equation}\label{eq:abp}
		(\eta(\lambda;A,B,P))^2 =
		\min_{x\in\bbS^{r+n}}
		\left(\frac{x^* G_1 x}{x^*x} + \frac{1}{\sum_{j=0}^d |\lambda|^{2j}}\cdot \frac{x^*
		G_2 x}{x^*H_2x}\right).
	\end{equation}

	\item \label{thm:3b:ibcp}
	$\eta(\lambda;B,C,P) <\infty$ and it satisfies 
	\begin{equation}\label{eq:bcp}
		(\eta(\lambda;B,C,P))^2 =
		\min_{x\in\bbS^{r+n}}
		\left(\frac{x^* G_1 x}{x^*H_2x} + \frac{x^* G_2 x}{x^*H_\lambda
		x}\right),
	\end{equation}
	where $H_{\lambda}:=H_1+(\sum_{j=0}^d |\lambda|^{2j})\cdot H_2$.

\end{enumerate}

\end{theorem}

\begin{proof}
	For item~\eqref{thm:3b:iabc}, it follows
	from~\eqref{eq:perror2inner} with $\Pi=\{A,B,C\}$ that 
	\begin{align*}
		(\varphi(x;A,B,C))^2 = 
		& \min~ {\|\dA\|}_F^2+{\|\dB\|}_F^2+ {\|\dC\|}_F^2\\
		& ~\text{s.t.}~~ \dA\in \C^{r,r}, \dB \in  \C^{r,n}, \dC \in \C^{n,r},\\
		& \text{~~~}~~~ \begin{bmatrix} A-\lambda I_r -\dA & B-\dB \\
		C-\dC & P(\lambda) \end{bmatrix}x=0.
	\end{align*}
	According to~\Cref{lem:3b}, with $\mu_1=\mu_2=1$, $d_1=d_2=0$, and
	$M=S(\lambda)$, we obtain 
	\[
		(\varphi(x;A,B,C))^2=
			\frac{\|[\begin{smallmatrix} (A-\lambda I_r) & B
			\end{smallmatrix}] x\|^2}{\|x_1\|^2+\|x_2\|^2} 
			+\frac{\|[\begin{smallmatrix} C  & P(\lambda) \end{smallmatrix}]x\|^2}
			{\|x_1\|^2},
	\]
	which leads to~\eqref{eq:abc} by recalling~\eqref{eq:perror2}.
	Moreover, by an $x$ with $x_1\neq 0$, we obtain
	$\eta(\lambda;A,B,C)\leq \varphi(x;A,B,C) < \infty$,

	For item~\eqref{thm:3b:iabp}, it follows
	from~\eqref{eq:perror2inner} with $\Pi=\{A,B,P\}$ that 
	\begin{align*}
		(\varphi(x;A,B,P))^2 = 
		& \min~ {\|\dA\|}_F^2+{\|\dB\|}_F^2+ \textstyle \sum_{j=0}^d{\|\Delta_j\|}_F^2\\
		& ~\text{s.t.}~~ \dA\in \C^{r,r}, \dB \in  \C^{r,n}, \Delta_j \in \C^{n,n},\text{ for $j=0,1,\dots,d$},\\
		& \text{~~~}~~~
		\begin{bmatrix}
			B-\dB & A-\lambda I_r -\dA   \\ 
			P(\lambda) -\sum_{j=0}^d \lambda^j\Delta_j & C 
		\end{bmatrix}
		\begin{bmatrix}x_2 \\ x_1 \end{bmatrix}=0,
	\end{align*}
	where the equality constraint 
	is obtained from $(S(\lambda)-\dS(\lambda))x=0$, 
	by switching the 
	column blocks of the matrix $S(\lambda)-\dS(\lambda)$.
	Then, by applying~\Cref{lem:3b} with $\mu_1=1$, $d_1=0$, $\mu_2=\lambda$, $d_2=d$, 
	the rest of the proof follows that of item~\eqref{thm:3b:iabc}.

	For item~\eqref{thm:3b:ibcp}, it follows
	from~\eqref{eq:perror2inner} with $\Pi=\{B,C,P\}$ that 
	\begin{align*}
		(\varphi(x;B,C,P))^2 = 
		& \min~ {\|\dB\|}_F^2+{\|\dC\|}_F^2+ \textstyle \sum_{j=0}^d{\|\Delta_j\|}_F^2\\
		& ~\text{s.t.}~~ \dB\in \C^{r,n}, \dC \in  \C^{n,r}, \Delta_j \in \C^{n,n},\text{ for $j=0,1,\dots,d$},\\
		& \text{~~~}~~~
		\begin{bmatrix} 
			P(\lambda) -\sum_{j=0}^d \lambda^j\Delta_j & 		C -\dC \\
			B-\dB    & A-\lambda I_r 
		\end{bmatrix}
		\begin{bmatrix}x_2 \\ x_1 \end{bmatrix}=0,
	\end{align*}
	where the equality constraint 
	is obtained from $(S(\lambda)-\dS(\lambda))x=0$, 
	by switching both the row and column blocks of the matrix $S(\lambda)-\dS(\lambda)$.
	Again, by applying~\Cref{lem:3b} with $\mu_1=\lambda $, $d_1=d$, $\mu_2=1$, $d_2=0$, 
	the rest of the proof follows that of item~\eqref{thm:3b:iabc}.
\end{proof}

\begin{remark}\label{rmk:etaacp}
\rm
The backward error $\eta(\lambda; A, C, P)$ of the Rosenbrock system
matrix $S(z)$, can be obtained from the transposed Rosenbrock system matrix
$\widehat S(z)\equiv (S(z))^T$ as given by~\eqref{eq:hats}.
By block symmetry, we have 
\begin{equation}\label{eq:acp}
\eta(\lambda;A,C,P) \equiv \widehat\eta(\lambda; \widehat A,\widehat B,\widehat P),
\end{equation}
where $\widehat\eta(\lambda; \widehat A,\widehat B,\widehat P)$
denotes the eigenvalue backward error of the new Rosenbrock system matrix
$\widehat S(z)$ subject to a perturbation pattern $\Pi=\{\widehat A,\widehat B, \widehat P\}$.
Now, $\widehat\eta(\lambda; \widehat A,\widehat B,\widehat P)$ 
can be evaluated by the formula~\eqref{eq:abp}
applied to the transposed system $\widehat S(z)$.
\end{remark}

%
\section{Minimization of Sum of Two Rayleigh Quotients}\label{sec:SRQ2NEPv}
Regardless of full and partial block perturbations,
the eigenvalue backward errors of the Rosenbrock system
matrix can be expressed as an optimization problem in the form of
\begin{equation}\label{eq:srq2min}
	\min_{x\in\mathbb C^n,\ \|x\|=1}~
	f(x):= \frac{x^*A_1x}{x^*(\alpha_1 I+\beta_1 A_3)x} 
	+ \frac{x^*A_2x}{x^*(\alpha_2 I+\beta_2 A_3)x},
\end{equation}
where $A_1,A_2,A_3\in\C^{n,n}$ are Hermitian and positive semidefinte
matrices,
$\alpha_i,\beta_i$ are given scalars satisfying 
$\alpha_i I+\beta_i A_3\succeq 0$,
for $i=1,2$.
For example, in the case of full perturbation
with $\eta(\lambda)$ given by~\eqref{eq:bs},
we can set $A_1=G_1$, $A_2=G_2$, $A_3=H_2$,
$\alpha_1=\alpha_2=1$, $\beta_1=0$, and $\beta_2=\gamma-1$, 
noticing that $H_1=I-H_2$.
For partial perturbations, 
a few cases (\Cref{thm:1b,thm:2b})
involve a single Rayleigh quotient in $f(x)$ (i.e., $A_2=0$),
but in general we must address two non-trivial Rayleigh quotients.

%
Since the objective function is the Sum of Two generalized Rayleigh Quotients, 
we call~\eqref{eq:srq2min} an SRQ2 minimization.
Recall that a Rayleigh quotient of the form $\frac{0}{0}$ is 
assigned a value of $0$ by convention~\eqref{eq:defrq}.
This ensures $f(x)$ in~\eqref{eq:srq2min} is lower semi-continuous 
(i.e., $\lim_{x\to x_0}f(x)\leq f(x_0)$ for $x_0\neq 0$)
and hence, the minimization~\eqref{eq:srq2min} is well-defined.
Clearly, the objective function $f(x)$ is differentiable at 
$x\in\C^n$ satisfying $x^*(\alpha_i I+\beta_i A_3)x\neq 0$, for $i=1,2$
(i.e., both denominators of $f(x)$ are nonzero).

The SRQ2 minimization~\eqref{eq:srq2min},
subject to the spherical constraint $\|x\|=1$, 
can be solved 
by conventional Riemannian optimization techniques,
such as the Riemannian trust-regions method; see, e.g.,~\cite{Absil:2009}.
However, those general-purpose Riemannian optimization methods 
often target local minimizers, and they 
do not fully exploit the special SRQ2 form of the objective
function $f(x)$ for analysis and computation.

In this section, we introduce a novel technique to 
solve the SRQ2 minimization~\eqref{eq:srq2min}.
We will recast the problem into an optimization 
over the joint numerical range of the Hermitian matrices $A_1,A_2,A_3$.
The new problem typically presents fewer
local minimizers than the original one, 
so it is more favorable 
for analysis and computation. 
The convexity in the joint numerical range
also allows for development of a nonlinear eigenvector approach to
efficiently solve the optimization problem, 
as well as a visualization technique to aid in verification of 
the global optimality of solution.

\subsection{Minimization Over Joint Numerical Range}
Let $\mathcal M:=(A_1,A_2,A_3)$ contain the coefficient matrices of the 
SRQ2 minimization~\eqref{eq:srq2min}.
We define the Joint Numerical Range (JNR) associated with $\mathcal M$ as
\begin{equation}\label{eq:jnr}
	W(\mathcal M):=
	\left\{\rhom(x) \colon x\in\mathbb C^n,\ \|x\|=1
	\right\},
\end{equation}
where $\rhom\colon \C^n\to \R^3$ consists of the quadratic forms
of the Hermitian matrices in $\mathcal M$: 
\begin{equation}\label{eq:rho}
	\rhom(x):= \left[ x^*A_1x,\ x^* A_2x,\ x^*A_3x\right]^T.
\end{equation}
It is well-known that the JNR $W(\mathcal M)$ is a closed and connected
region in $\R^3$, 
and it is a convex set if the size $n$ of the matrices $A_i$ satisfies $n\geq 3$; 
see, e.g.,~\cite{Au:1983,Muller:2020}.

In the SRQ2 minimization~\eqref{eq:srq2min}, 
the objective function $f(x)$ is
defined through the three quadratic forms $\{x^*A_ix\}$, for $i=1,2,3$.  
Hence, we can write 
\begin{equation}\label{eq:compf}
	f(x)=g(\rhom(x)),
\end{equation}
where $g\colon \R^3\to \R$ is given by 
\begin{equation}\label{eq:gfun}
	g(y) := \frac{y_1}{\alpha_1+\beta_1y_3}+\frac{y_2}{\alpha_2+\beta_2
	y_3},
\end{equation}
and $y\equiv [y_1,y_2,y_3]^T\in\R^3$.
Consequently, by using an intermediate variable $y=\rhom(x)$, 
we recast the SRQ2 minimization~\eqref{eq:srq2min} into
\begin{equation}\label{eq:optw}
	\min_{y\in W(\mathcal M)} g(y),
\end{equation}
where the feasible set $W(\mathcal M)$, the JNR as defined in~\eqref{eq:jnr},
contains all possible values of $y=\rhom(x)$ over 
$x\in\bbS^n$.
We call problem~\eqref{eq:optw} the {\em JNR minimization} 
associated with the SRQ2 minimization~\eqref{eq:srq2min}.

In the JNR minimization~\eqref{eq:optw},
the variable $y\in\mathbb R^3$ is a real and $3$-dimensional vector, 
as in contrast to a complex and $n$-dimensional variable vector $x\in\mathbb C^n$
of the original SRQ2 minimization~\eqref{eq:srq2min}.
Typically, $n \gg 3$ in practice, so the new feasible region
$W(\mathcal M)$ for $y$ is a convex set.
This convexity facilitates analysis, computation,
and visualization the JNR minimization~\eqref{eq:jnr},
as to be shown shortly.

Since the map $\rhom\colon \C^n\to \R^3$ is 
generally not bijective,
the two optimization problems~\eqref{eq:srq2min} and~\eqref{eq:optw}
are, however, not strictly {\em equivalent}. 
The relationship of the solution for the two problems are summarized in~\Cref{lem:relation}.
Recall that by the standard notion in optimization,
a general optimization problem $\min\{ F(x)\colon x\in
\Omega\}$ has a global minimizer $x_{\star}\in\Omega$, if it holds 
$F(x) \geq F(x_{\star})$ for all $x\in \Omega$.
It has a local minimizer $x_{\star}\in\Omega$,
if $F(x) \geq F(x_{\star})$ for all $x\in \Omega$ with
$x$ sufficiently close to $x_{\star}$.

\begin{lemma}\label{lem:relation}
	Consider the optimization problems~\eqref{eq:srq2min} and~\eqref{eq:optw}.
	\begin{enumerate}[(a)]
		\item \label{lem:relation:a}
		A vector
		$x_{\star}\in\bbS^n$ is a global minimizer of SRQ2 minimization~\eqref{eq:srq2min},
		if and only if $y_{\star}=\rhom(x_{\star})$ is a global
		minimizer of the JNR minimization~\eqref{eq:optw}.
		\item \label{lem:relation:b}
		If $y_{\star}\in\mathbb R^3$ is a local minimizer of the JNR
		minimization~\eqref{eq:optw},
		then any $x_{\star}\in\bbS^n$ with 
		$\rhom(x_{\star})=y_{\star}$ is a local 
		minimizer of the SRQ2 minimization~\eqref{eq:srq2min}.
	\end{enumerate}
\end{lemma}
\begin{proof}
	Both items~\eqref{lem:relation:a} and~\eqref{lem:relation:b}
	are direct consequences of the parameterization $y=\rhom(x)$,
	which implies $f(x)=g(y)\geq g(y_{\star}) = f(x_{\star})$ as $x\to x_{\star}$.
\end{proof}

We stress that the reverse of~\Cref{lem:relation}~\eqref{lem:relation:b} 
is not necessarily true. 
Namely, if $x_{\star}\in\bbS^n$ is a local minimizer of the SRQ2
minimization~\eqref{eq:srq2min}, then $y_{\star}=\rhom(x_{\star})$ 
may not be a local minimizer of the JNR minimization~\eqref{eq:optw}.
This implies the JNR minimization~\eqref{eq:optw}
may have much fewer local minimizers than the original SRQ2
minimization~\eqref{eq:srq2min}.
This is a desirable property,
since the local optimization technique will have a better chance 
to find the global optimal solution.

\subsection{Local Optimality Condition and NEPv}
By exploiting the convexity in the JNR,
we will establish a variational characterization for the local minimizers
of the JNR minimization~\eqref{eq:optw}.
This characterization can also be equivalently expressed as 
a nonlinear eigenvalue problem with eigenvector dependency (NEPv).

\begin{theorem}\label{thm:nepvnr}
	Let $A_i\in\C^{n,n}$, for $i=1,2,3$, be Hermitian matrices of size
	$n\geq 3$, 
	$\mathcal M=(A_1,A_2,A_3)$,
	$W(\mathcal M)$ be as defined in~\eqref{eq:jnr},
	and 
	$g$ be as given by~\eqref{eq:gfun}.
    Suppose that $y_{\star}$ is a local minimizer of $g$ in $W(\mathcal M)$,
    and $g$ is differentiable at $y_{\star}$,
    then 
    \begin{equation}\label{eq:nepvnr}
    	\min_{y\in W(\mathcal M)} \nabla g(y_{\star})^T y =
    	\nabla g(y_{\star})^Ty_{\star}.
    \end{equation}
	Moreover,~\eqref{eq:nepvnr}
	holds if and only if $y_{\star}=\rhom(x_{\star})$
    for an $x_{\star}\in\bbS^n$ satisfying the NEPv:
    \begin{equation}\label{eq:nepv}
    	H(x)x= \lambda x,
    \end{equation}
    where $ H(x)\in\C^{n,n}$ is a Hermitian matrix given by 
    \begin{equation}\label{eq:hx}
    	 H(x):= \sum_{i=1,2}
        \left[
         \frac{1}{\alpha_i + \beta_i\cdot x^*A_3x}  A_i
         - \frac{(x^*A_ix)\cdot \beta_i} {(\alpha_i + \beta_i\cdot x^*A_3x)^2} A_3
        \right],
    \end{equation}
    and $\lambda$ is the smallest eigenvalue of $ H(x)$.
\end{theorem}

\begin{proof}
	We first show equation~\eqref{eq:nepvnr}.
	Let $y$ be an arbitrary point in $W(\mathcal M)$. 
	Recall that the JNR $W(\mathcal M)$ is a convex set 
	for matrices of size $n\geq 3$ (see, e.g.,~\cite{Au:1983}). 
	By convexity, $W(\mathcal M)$ contains the entire line segment
	between $y$ and $y_{\star}$.
	Namely,
	\[
	\text{$y(t):=y_{\star}+t(y-y_{\star})\in W(\mathcal M)$ for all $t\in[0,1]$.}
	\]
	Since $g$ is differentiable at $y_{\star}$, 
	by Taylor's expansion of $g(y)$ at $y_{\star}$, 
	\[
	g(y(t)) = g(y_{\star}) + t\cdot \nabla g(y_{\star})^T (y-y_{\star}) + \mathcal O(
	t^2\|y-y_{\star}\|^2).
	\]
	The local minimality of $y_{\star}$ implies $g(y(t))\geq g(y_{\star})$ for 
	$t>0$ sufficiently small,
	so it holds that $\nabla g(y_{\star})^T (y-y_{\star}) \geq 0$,
	i.e.,
	\[
	\nabla g(y_{\star})^Ty_{\star} \leq \nabla g(y_{\star})^T y.
	\]
	Since the inequality above holds for all $y\in W(\mathcal M)$, we can derive 
	\[
	\nabla g(y_{\star})^Ty_{\star}\leq  
	\min_{y\in W(\mathcal M)} \nabla g(y_{\star})^T y \leq
	\nabla g(y_{\star})^Ty_{\star},
	\]
	where the last inequality is by $y_{\star}\in W(\mathcal M)$.
	As equality must hold, we obtain~\eqref{eq:nepvnr}.

	Next we show $y_{\star}=\rhom(x_{\star})$ for some $x_{\star}\in\bbS^n$ 
	satisfying the NEPv~\eqref{eq:nepv}. 
	Since $y$ and $y_{\star}\in W(\mathcal M)$, 
	we can write $y=\rhom(x)$ and $y_{\star}=\rhom(x_{\star})$ for some
	$x,x_{\star}\in\bbS^n$, according to the definition in~\eqref{eq:jnr}.
    We can derive
    \begin{equation}\label{eq:gty}
    	\nabla g(y_{\star})^T y = 
    	\nabla g(\rhom(x_{\star}))^T\cdot \rhom(x) 
    	= x^*{H}(x_{\star}) x,
    \end{equation}
    where in the last equality we used the gradient 
	of the function $g$ defined in~\eqref{eq:gfun}: 
    \[
    \nabla g(y) = 
    \left[
    \frac{1}{\alpha_1+\beta_1y_3},\ 
    \frac{1}{\alpha_2+\beta_2y_3},\
    \sum_{i=1,2} -\frac{\beta_iy_i}{(\alpha_i+\beta_iy_3)^2},
    \right]^T.
    \]
	Using~\eqref{eq:gty}, we can write the characterization~\eqref{eq:nepvnr} as
    \[
    \min_{x\in\C^n,\ \|x\|=1} x^* H(x_{\star}) x =x_{\star}^*  H(x_{\star}) x_{\star}.
    \]
    According to the well-known minimization principle of the Hermitian eigenvalue
    problems, the minimal value of the quadratic form $x^* H(x_{\star}) x$
    from above is achieved at the eigenvector of the smallest eigenvalue of 
    $ H(x_{\star}) x =\lambda x$.
    Consequently, $x_{\star}$ is the eigenvector for the smallest eigenvalue of the
    Hermitian matrix $ H(x_{\star})$. We proved~\eqref{eq:nepv}.
\end{proof}


By~\Cref{thm:nepvnr}, the NEPv~\eqref{eq:nepv} and the equivalent variational
equation~\eqref{eq:nepvnr} provide local optimality condition for the JNR
minimization~\eqref{eq:optw}.
We note that those conditions are {\em stronger} than the standard 
first-order optimality condition for the original SRQ2
minimization~\eqref{eq:srq2min},
for which, by introducing the Lagrange function
$L(x,\lambda)=f(x) - \lambda(x^*x-1)$
and set $\nabla_x L=0$ 
(assuming $x\in\R^n$ is a real vector so that $L$ is differentiable),
we obtain the same equation
as~\eqref{eq:nepv} but do not require $\lambda$ to be the smallest
eigenvalue of $ H(x)$.
This advantage of NEPv characterization is partially explained by the relation
of the two optimization problems revealed in~\Cref{lem:relation}.

In the literature, NEPv characterizations have been explored in various
optimization problems with orthogonality constraints; 
see~\cite{Bai:2024} and references therein. 
Those characterizations allow for efficient solution of the optimization
problem by exploiting state-of-the-art eigensolvers.
Particularly, in~\cite{Zhang:2013,Zhang:2014} the authors
proposed NEPv characterizations that apply to the optimization 
of sum of Rayleigh quotients:
\begin{equation}\label{eq:srq2s}
	\min_{x\in\mathbb C^n,\ \|x\|=1}~
	f(x):= x^*A_1x + \frac{x^*A_2x}{x^*A_3x},
\end{equation}
where $A_i\in\C^{n,n}$ for $i=1,2,3$ are Hermitian
positive definite matrices\footnote{
	The works~\cite{Zhang:2013,Zhang:2014} focus on
	maximization of $f(x)$ rather than minimization, and they assume $A_3\succ 0$. 
	These problems are equivalent to~\eqref{eq:srq2s}
	by negating and shifting the objective function to $-f(x)+\sigma$, with $\sigma>0$.
	Additionally, they only consider real $A_i\in\R^{n,n}$ and $x\in\R^n$. 
	However, we can always represent the variable $x=a+b\imath\in\C^n$ by its real and imaginary parts 
	as $[a^T,b^T]^T\in\R^{2n}$, thereby converting~\eqref{eq:srq2s} into
	an equivalent real form.
}.
Our development of the NEPv differs from 
the previous works~\cite{Zhang:2013,Zhang:2014} in several aspects. 
First, the SRQ2 minimization~\eqref{eq:srq2min}
cannot generally be reduced to the form of~\eqref{eq:srq2s},
due to the presence of semidefinite matrices in both denominators of 
the objective function $f(x)$ in~\eqref{eq:srq2s}.
Therefore, the previous analyses~\cite{Zhang:2013,Zhang:2014}
do not directly apply to the SRQ2 minimization~\eqref{eq:srq2min}.
Second, the developments in~\cite{Zhang:2013,Zhang:2014} primarily rely on the 
KKT conditions and properties of the {\em global optimal} solution
of the optimization problem~\eqref{eq:srq2s}.
In contrast, by exploiting the JNR
minimization~\eqref{eq:optw},
our \Cref{thm:nepvnr} reveals the nature of NEPv~\eqref{eq:nepv} as the
first-order {\em local optimality} condition of the JNR
minimization~\eqref{eq:optw}.
Combined with~\Cref{lem:mapping}, this insight explains why 
NEPv~\eqref{eq:nepv} can characterize global optimality
but not local optimality of the minimization problem~\eqref{eq:srq2s},
advancing the previous observations in~\cite{Zhang:2013,Zhang:2014}.
Additionally, our use of the JNR significantly simplifies 
the derivation of NEPv for~\eqref{eq:srq2s}.
Particularly, by a straightforward extension, it provides a unified treatment 
for problem~\eqref{eq:srq2s} and its block extension given by:
Minimize $f(X):=\mbox{tr}(X^*A_1X) + \mbox{tr}(X^*A_2X) / \mbox{tr}(X^*A_3X)$ 
subject to $X\in\C^{n,k}$ and $X^*X=I_k$.
Similar trace minimizations were considered in~\cite{Zhang:2014}, 
but over the real Stiefel manifold with variables $X\in\R^{n,k}$ and $X^TX=I_k$.



%

\subsection{Non-differentiable Minimizers}\label{sec:nondiff}
\Cref{thm:nepvnr} requires the
local minimizer $y_{\star}$ to be differentiable for the 
rational function $g$, but this is not a significant restriction.
The non-differentiability issue can occur if one of the 
fractions in $g(y_{\star})$, as given by~\eqref{eq:gfun}, 
takes the form $\frac{0}{0}$.
We will show that such non-differentiable local minimizers can be 
quickly identified by Hermitian eigenvalue problems.

For simplicity, we assume the two positive semidefinite matrices 
in the denominators of the objective function $f(x)$ of the SRQ2
minimization~\eqref{eq:srq2min}, 
$(\alpha_i I+\beta_iA_3)$ for $i=1,2$,
do not share any common null vectors. 
Equivalently, this is expressed as
\begin{equation}\label{eq:comnull}
	(\alpha_1 I+\beta_1A_3) +(\alpha_2 I+\beta_2A_3) \succ 0,
\end{equation}
given that $(\alpha_i I + \beta_i A_3)\succeq 0$.
The condition~\eqref{eq:comnull} holds for all SRQ2 minimizations
from the backward error analysis of Rosenbrock systems discussed in~\Cref{preliminaries}.
Moreover, to determine the differentiability of $g(y)$,
we introduce the Hermitian matrices
\begin{equation}\label{eq:qi}
	Q_i := A_i+(\alpha_i I + \beta_i A_3),
	\text{ for $i=1,2$,}
\end{equation}
corresponding to each Rayleigh quotient in the 
function $f(x)$ given by~\eqref{eq:srq2min}. 
Recalling that $A_i\succeq 0$ and 
$(\alpha_i I + \beta_i A_3)\succeq 0$,
we have $Q_i\succeq 0$.
A quick verification shows that $Q_i\succ 0$  if and only if 
the Rayleigh quotient $\frac{x^*A_ix}{x^*(\alpha_i I+\beta_i A_3)x}$ won't encounter the form 
$\frac{0}{0}$.

\begin{theorem}\label{thm:nondiff}
	Let assumption~\eqref{eq:comnull}  hold for the 
	JNR minimization~\eqref{eq:optw}
	and 
	$Q_i$ be given by~\eqref{eq:qi}.
	Let $y_{\star}$ be a local minimizer of $g$ over 
	$W(\mathcal M)$ with $g(y_{\star})<\infty$.
	\begin{enumerate}[(a)]
		\item \label{thm:nondiff:ia}
		If $Q_1\succ 0$ and $Q_2\succ 0$, then 
		$y_{\star}$ must be a differentiable point for $g(y)$,
		and hence, it admits the NEPv characterization given by~\Cref{thm:nepvnr}.

		\item \label{thm:nondiff:ib}
		Otherwise, $y_{\star}$ may be non-differentiable for $g(y)$,
		and then it must be given by 
		$y_{\star}=\rhom(K_iv_i)$ for $i=1$ or $i=2$,
		where $K_i$ is an orthogonal basis matrix of the null space 
		of $Q_i$, and $v_i$ denotes an eigenvector for the smallest eigenvalue
		$\mu_i$ of the following Hermitian eigenvalue problems:
		\begin{subequations}\label{eq:nondiffeig}
		\begin{align}
			(K_1^*A_2K_1)\cdot v_1&=\mu_1 \cdot (\alpha_2I+\beta_2 K_1^*A_3K_1)\cdot v_1,  \label{eq:nondiffeig1}\\
			(K_2^*A_1K_2)\cdot v_2&=\mu_2 \cdot (\alpha_1I+\beta_1 K_2^*A_3K_2)\cdot v_2, \label{eq:nondiffeig2}
		\end{align}
		\end{subequations}
		where the two Hermitian matrices on the right-hand side 
		are positive definite.
	\end{enumerate}
\end{theorem}

\begin{proof}
By~\eqref{eq:gfun}, we can write the objective function as $g=g_1+g_2$ with 
\begin{equation}\label{eq:gi}
	g_i(y) := 
	\frac{y_i}{\alpha_i + \beta_i y_3}
	\equiv 
	\frac{x^*A_ix}{x^*(\alpha_iI + \beta_i A_3)x},
\end{equation}
for $i=1,2$, where we parameterize $y=\rhom(x)$
by an $x\in\bbS^n$.
Clearly, $g_i$ is differentiable at $y$
if and only if the denominator
$\alpha_i + \beta_i y_3\equiv x^*(\alpha_iI + \beta_i A_3)x\neq 0.$

For item~\eqref{thm:nondiff:ia}, 
by parameterizing $y_{\star}=\rhom(x_{\star})$,
we must have $x_{\star}^*(\alpha_iI + \beta_i A_3)x_{\star} \neq 0$
for $i=1,2$.
Because otherwise, we have $x_{\star}^*A_ix_{\star}\equiv x^*Q_ix>0$.
Then by~\eqref{eq:gi}, $g_i(y_{\star})= (x^*Q_ix)/0 =\infty$, 
contradicting $g(y_{\star}) < \infty$.
Hence, $g$ is differentiable at $y_{\star}$.

For item~\eqref{thm:nondiff:ib}, suppose that $g_1$ 
is not differentiable at $y_{\star}=\rhom(x_{\star})$,
i.e., its denominator in definition~\eqref{eq:gi} vanishes: 
$ \alpha_i + \beta_i y_3 \equiv x_{\star}^*(\alpha_1I + \beta_1
A_3)x_{\star}=0$.
Since $g_1(y_{\star})<\infty$, the numerator in $g_1(y_{\star})$ 
must also vanish: $y_i\equiv x_{\star}^*A_1x_{\star}= 0$. 
So, we have
\begin{equation}\label{eq:nullx}
x_{\star}\in 
\mbox{null}(A_1)\bigcap \mbox{null}(\alpha_1 I + \beta_1 A_3)
\equiv 
\mbox{null}(A_1+\alpha_1 I + \beta_1 A_3)
\equiv 
\mbox{range}(K_1), 
\end{equation}
where we used $A_1\succeq 0$ and $(\alpha_1 I + \beta_1 A_3)\succeq 0$
and $K_1$ is an orthogonal basis matrix.

Now, by the local minimality of $y_{\star}$ and~\Cref{lem:relation},
we have $x_{\star}$ must be a local minimizer of $f(x)$ 
with $x\in\bbS^n$ restricted to the subspace $\mbox{range}(K_1)$:
\begin{equation}\label{eq:localmin}
\min_{\substack{x=K_1v\\ v\in\bbS^p}} 
f(x)
=
\min_{\substack{x=K_1v\\ v\in\bbS^p}} 
\frac{x^*A_2x}{x^*(\alpha_2I_n+\beta_2 A_3)x}
= 
\min_{\substack{x=K_1v\\ v\in\bbS^p}} 
\frac{v^*(K_1^*A_2K_1)v}{v^*(\alpha_2I+\beta_2 K_1^*A_3K_1)v},
\end{equation}
where the first equality is due to $f(x)\equiv g_1(\rhom(x))+g_2(\rhom(x))$
and $g_1(\rhom(x)) = \frac{0}{0} \equiv 0$ 
for all $x\in \mbox{range}(K_1)$
(recalling~\eqref{eq:nullx} and our convention~\eqref{eq:defrq}).
Moreover, since the assumption~\eqref{eq:comnull}
implies that $(\alpha_iI_n+\beta_i A_3)$
for $i=1,2$ do not have a common null vector,
we have $(\alpha_2I+\beta_2 K_1^*A_3K_1)\succ 0$
for the matrix in the denominator of~\eqref{eq:localmin}.

As a well established result in matrix analysis 
and optimization, 
a local minimizer of a Rayleigh quotient must be its global 
minimizer; see, e.g.,~\cite[Prop 4.6.2]{Absil:2009}.
By~\eqref{eq:localmin} and Hermitian eigenvalue 
minimization principle, 
$x_{\star}=K_1v_1$, where
$v_1$ is the eigenvector for the smallest 
eigenvalue of the Hermitian eigenvalue problem~\eqref{eq:nondiffeig1}.

Analogously, if $g_2$ is not differentiable at $y_{\star}$, then we have
$y_{\star}=\rhom(K_2v_2)$, where $v_2$ is an eigenvector for the smallest 
eigenvalue of the eigenvalue problem~\eqref{eq:nondiffeig2}.
$\qed$
\end{proof}

In certain cases of partial block perturbations, as described 
in~\Cref{thm:2b,thm:3b},
the corresponding SRQ2 minimization may not have
positive definite $Q_1$ or $Q_2$ matrices.
To find the minimal solution of these problems, we can solve 
NEPv~\eqref{eq:nepv}, provided it has a solution,
along with the Hermitian eigenvalue 
problems~\eqref{eq:nondiffeig}. We can then 
select the $x_{\star}$ with the minimal object value $f(x_{\star})$ 
as the solution.

\section{Numerical Experiments}\label{sec:example}
We present numerical examples of SRQ2 minimization~\eqref{eq:srq2min}
for backward error analysis of the Rosenbrock system matrix.
Our primary goal is to demonstrate the benefits 
of the JNR reformulation~\eqref{eq:optw}, 
including elimination of certain local 
minimizers of~\eqref{eq:srq2min},
efficient solution via the NEPv characterization~\eqref{eq:nepv},
and visualization of the solution.
All experiments are implemented in MATLAB and run on a MacBook Pro laptop
with an M1 Pro CPU and 32G memory.
Codes and data are available through
\texttt{https://github.com/ddinglu/minsrq2}.

To solve NEPv~\eqref{eq:nepv}, we apply the 
level-shifted Self-Consistent-Field (SCF)
iteration~\cite{Bai:2024,Yang:2007,Zhang:2014a}:
Starting from an initial guess $x_0\in\bbS^n$, iteratively solve
Hermitian eigenvalue problems
\begin{equation}\label{eq:lsscf}
[H(x_k) + \sigma_k x_kx_k^* ] x_{k+1} = \lambda_{k+1} x_{k+1},
\end{equation}
for $k=0,1,\dots$, where $\sigma_k\in \mathbb R$ is a given level-shift, and $x_{k+1}$ is
the eigenvector corresponding to the smallest eigenvalue $\lambda_{k+1}$ of
the Hermitian matrix $H(x_k) + \sigma_k x_kx_k^*$.
If $\sigma_k\equiv 0$, then~\eqref{eq:lsscf} reduces to the plain SCF iteration.
But a non-zero level-shift $\sigma_k$ can help to stabilize and accelerate the
convergence of plain SCF. 
Particularly, a plain SCF may not always converge,
whereas for sufficiently large $\sigma_k$, level-shifted SCF is guaranteed
locally convergent under mild assumptions; see, e.g.,~\cite{Bai:2024}.
In our implementation, we select level-shifts $\sigma_k$ adaptively by sequentially trying 
\begin{equation}\label{eq:lss}
\sigma_k = 0,\, 2\delta_k,\, 2^2\delta_k,\, 2^3\delta_k,\dots,
\end{equation}
until a reduction in the objective function $f(x)$ or 
residual norm of NEPv is achieved by $x_{k+1}$.
Here, $\delta_k = \lambda_{2}(H(x_k)) - \lambda_{1}(H(x_k))$ is the 
eigenvalue gap between the first and second smallest eigenvalue of
$H(x_k)$,
and the use of $\sigma_k=2\delta_k$ is a common practice;
see, e.g.,~\cite{Yang:2007,Zhang:2014a}.
Often, a plain SCF ($\sigma_k=0$)
can produce a reduced $f(x_{k+1})$,
and then there is no need to 
enter the selection loop~\eqref{eq:lss} for $\sigma_k$.
Finally, the level-shifted SCF~\eqref{eq:lss} is considered to have
converged if the relative residual norm satisfies
\begin{equation}\label{eq:reltol}
\frac{\|H(x_k) x_k - s_k x_k\|}{\|H(x_k)\|_1 + 1}\leq 10^{-10}
\quad\text{with}\quad
s_k = x_k^*H(x_k) x_k,
\end{equation}
and $\|\cdot\|_1$ denotes the matrix 1-norm
(i.e., maximal absolute column sum).
We mention that the performance of SCF can be further enhanced by 
various acceleration techniques, such as damping, Anderson acceleration,
preconditioning, etc; see, e.g.,~\cite{Bai:2022} and references therein.
But those techniques are beyond the scope of this paper. 

To solve SRQ2 minimization~\eqref{eq:srq2min} directly,
we apply the Riemannian Trust-Region (RTR) method~\cite{Absil:2009}
-- 
a state-of-the-art algorithm for optimization problems with orthogonality
constraints and available in the software package \texttt{MANOPT}~\cite{Boumal:2014}.
The stopping criteria of RTR is based on the norm of 
the Riemannian gradient,
which coincides with twice the residual norm 
$\|H(x_k) x_k - s_k x_k\|$ with $s_k = x_k^*H(x_k) x_k$;
see, e.g.,~\cite{Absil:2009}.
For consistency with~\eqref{eq:reltol} in algorithm comparison, 
the error tolerance of RTR will be set as twice the residual 
norm of the last iteration by level-shifted SCF.

For demonstration, we compute the convex JNR $W(\mathcal M)\subset\R^3$
by sampling its boundary points.
A boundary point $y_v$ of $W(\mathcal M)$ 
with a given outer normal direction
$v\in\R^3$ is computed by $y_v=\rhom(x_v)$, where $x_v$ the eigenvector for 
the smallest eigenvalue of Hermitian eigenvalue problem
\begin{equation}\label{eq:supeig}
	H_vx=\mu x
	\qquad\text{and}\quad
	H_v=v_1A_1+v_2A_2+v_3A_3.
\end{equation}
We refer to~\cite{Bai:2024,Johnson:1978} for details on the computation of JNR.
To show a number $c$ is the global minimal value of $g$ over
$W(\mathcal M)$, 
we can check visually if the entire JNR $W(\mathcal M)$ would
lie on one side of the level surface $g(y)=c$.

\begin{example} \label{eg:jnr}
{\rm
	This example is to show that the JNR minimization~\eqref{eq:optw} 
	can avoid certain local 
	minimizers of the original SRQ2 minimization~\eqref{eq:srq2min}.
	To demonstrate,
	we consider an SRQ2 minimization of size $n=3$ in the following
	form 
	\begin{equation}\label{eq:eg1}
		\min \left\{f(x):=\frac{x^*A_1x}{x^*x} + \frac{x^*A_2x}{x^*A_3x} \colon x\in\C^n, \|x\|=1\right\},
	\end{equation}
	where the coefficient matrices are randomly generated and given by
	\[
	A_1=
	\left[\begin{array}{rrr}
		0.64 &  -0.15 &  -0.38\\
		-0.15 &   0.60 &  -0.22\\
		-0.38 &  -0.22 &   0.56
	\end{array}\right],
	\quad
	A_2=
	\left[\begin{array}{rrr}
		0.73 &    0.24 &   -0.07\\
		0.24 &    0.52 &   -0.04\\
		-0.07 &   -0.04 &    0.38
	\end{array}\right],
	\]
	and $A_3=\mbox{diag}( 0.53, 0.97, 0.38)$.
	Solving SRQ2 minimization~\eqref{eq:eg1} by the Riemannian trust-region
	method with various starting vectors, we found two optimal solutions 
	\begin{equation}\label{eq:eg1sol}
		x_{\star}^{(1)} = [-0.1728, -0.7704, -0.6137]^T
		\ \text{and}\ 
		x_{\star}^{(2)} = [-0.5730, 0.6282, -0.5263]^T.
	\end{equation}
	Both are local minimizers of SRQ2 minimization~\eqref{eq:eg1},
	since the Hessian matrices for $f(x)$ at $x_{\star}$
	can be verified positive definite over the complement of $x_{\star}$
	-- a sufficient condition for 
	local optimality of the constrained minimization~\eqref{eq:eg1};
	see, e.g.,~\cite{Nocedal:2006,Zhang:2013}.

	
	%
	By a change of variable $y=\rhom(x)$, the SRQ2 minimization~\eqref{eq:eg1} 
	can be written as the following JNR minimization problem
	\begin{equation}\label{eq:eg1b}
		\min \left\{ g(y):=y_1 + \frac{y_2}{y_3} \colon y\in W(\mathcal M)
		\right\},
	\end{equation}
	where $\mathcal M = (A_1,A_2,A_3)$.
	Here, we recall~\eqref{eq:jnr} for the definition of the joint numerical range $W(\mathcal M)$.
	The local minimizers $x_{\star}$ 
	from~\eqref{eq:eg1sol} yield $y_{\star}=\rhom(x_{\star})$ as given by 
	\begin{equation}\label{eq:ystar}
		y_{\star}^{(1)} = [0.2575,    0.4848,    0.7346]^T
		\quad\text{and}\quad
		y_{\star}^{(2)} = [0.6263,    0.3616,    0.6621]^T.
	\end{equation}
	Those optimal solutions are depicted in \Cref{fig:localmin},
	thanks to that the JNR minimization~\eqref{eq:eg1b} is over 
    $\R^3$, which allows for convenient visualization.
	The joint numerical range $W(\mathcal M)$ lies
	on one side of the level surface $g(y) = g(y_{\star}^{(1)})$,
	indicating $y_{\star}^{(1)}$ is (at least) a local minimizer of JNR
	minimization~\eqref{eq:eg1b}.
	Whereas the level surface $g(y) = g(y_{\star}^{(2)})$ cuts through
	$W(\mathcal M)$,
	implying $y_{\star}^{(2)}$ is not a local minimizer.
	That is to say, 
	despite both $x_{\star}^{(1)}$ and $x_{\star}^{(2)}$ are local minimizers of SRQ2
	minimization~\eqref{eq:eg1},
	only $y_{\star}^{(1)}=\rhom(x_{\star}^{(1)})$ occurs as a local minimizer of JNR minimization~\eqref{eq:eg1b}.

	\begin{figure}
		\centering
		\includegraphics[width=0.45\textwidth]{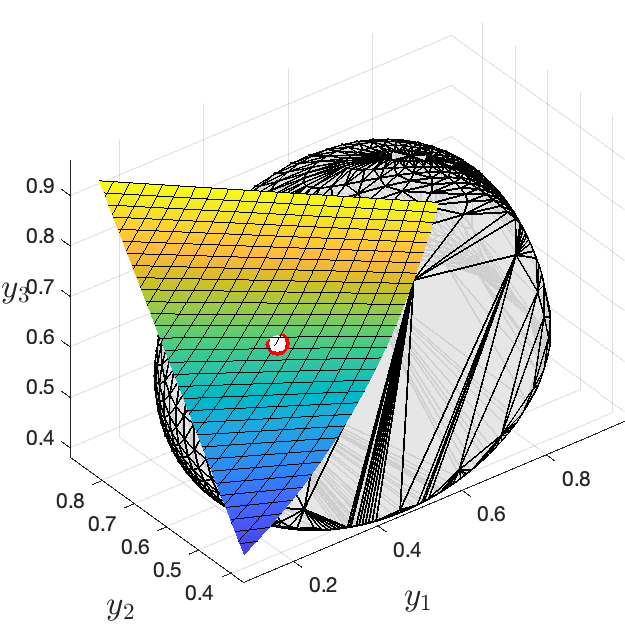}~\includegraphics[width=0.45\textwidth]{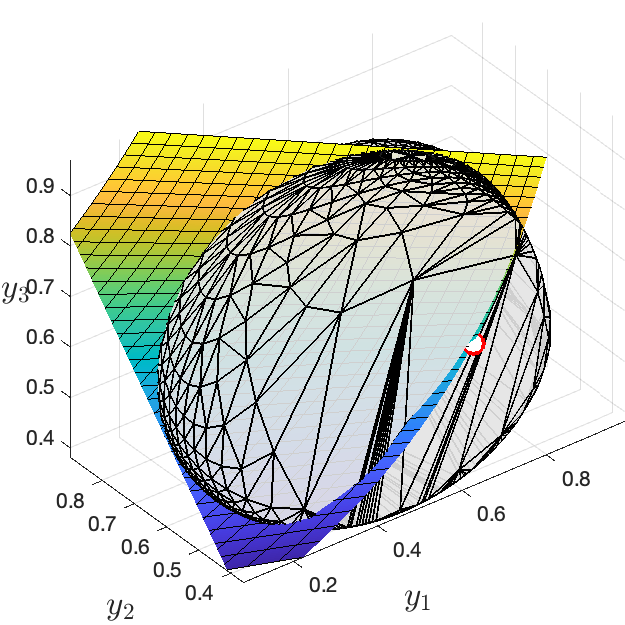}
		\caption{
			Joint numerical range $W(\mathcal M)$
			(bounded by gray surface, based on $N=800$ sample boundary points)
			and level surface $\{y\colon g(y) = g(y_{\star})\}$ 
			(colored) at solution $y_{\star}$ (marked `o')
			of optimization problem~\eqref{eq:eg1b}.
			The left plot is with $y_{\star}=y_{\star}^{(1)}$ and the right with
			$y_{\star}=y_{\star}^{(2)}$, both from~\eqref{eq:ystar}.
		} \label{fig:localmin}
	\end{figure}

	By~\Cref{thm:nepvnr}, since $y_{\star}^{(1)}=\rhom(x_{\star}^{(1)})$
	is a local minimizer of JNR minimization~\eqref{eq:eg1b},
	the corresponding $x_{\star}^{(1)}$ must be a solution to the
	NEPv~\eqref{eq:nepv}.  
	Numerically, 
	$$\|H(x_{\star}^{(1)})x_{\star}^{(1)} - \lambda_{\min}(H(x_{\star}^{(1)}))\cdot x_{\star}^{(1)}\| \approx 1.0\times 10^{-10},$$
	where $\lambda_{\min}$ is the smallest eigenvalue of $H(x_{*1})$, indicating 
	$x_{\star}^{(1)}$ is an approximate eigenvector of NEPv~\eqref{eq:nepv}.
	In contrast,  we have for the other solution $x_{\star}^{(2)}$ that 
	$$\|H(x_{\star}^{(2)})x_{\star}^{(2)} - \lambda_{\min}(H(x_{\star}^{(2)}))\cdot x_{\star}^{(2)}\| \approx 2.0\times 10^{-1}, $$
	indicating $x_{\star}^{(2)}$ is not an eigenvector of NEPv~\eqref{eq:nepv}.
	This is consistent to the fact that $y_{*2}$ is not a local minimizer
	of the JNR minimization~\eqref{eq:eg1b}.
	Indeed, we have $\|H(x_{\star}^{(2)})x_{\star}^{(2)} - \lambda_{2}(H(x_{\star}^{(2)})) \cdot x_{\star}^{(2)}\|
	\approx 10^{-10}$; i.e., $x_{\star}^{(2)}$ is with the
	second smallest eigenvalue $\lambda_2$ of $H(x_{\star}^{(2)})$,
	rather than the smallest one as required by NEPv~\eqref{eq:nepv}.
	Hence, solving NEPv~\eqref{eq:nepv} can avoid the local minimizer $x_{\star}^{(2)}$
	of SRQ2 minimization~\eqref{eq:eg1}.

	%
	Finally, it can be verified that the  approximate joint numerical range in \Cref{fig:localmin} 
	lies {\em entirely} on
	one side of the level surface,
	justifying pictorially
	that the solution $y_{\star}^{(1)}$ is a global minimizer of~\eqref{eq:eg1b}.
	Here, the approximate $W(\mathcal M)$ is
	constructed by the convex hull of $N=800$ sampled boundary points of $W(\mathcal M)$
	with uniformly distributed normal directions.
	Each boundary point is computed by 
	solving a Hermitian eigenvalue problems of size $n=3$
	as given by~\eqref{eq:supeig}.
	See~\cite{Bai:2024} for details of the computation.
} %
\end{example}

%
\begin{example}\label{eg:linsys}
{\rm
	For some specialized SRQ2 minimization in the form of~\eqref{eq:srq2s},
	it is known that NEPv characterization, as the one given by~\eqref{eq:nepv},
	allows for the use of state-of-the-art eigensolvers for fast solution of the
	problem;  see, e.g.,~\cite{Zhang:2013,Zhang:2014a}.
	In this example, we consider SRQ2 minimization~\eqref{eq:srq2min}
	intended for computing eigenvalue backward error of the Rosenbrock systems,
	and demonstrate the performance of NEPv approaches.
	For experiment, we consider a Rosenbrock system matrix with a 
	linear $P(z) = A_0+A_1z$, as given by
	\begin{equation}\label{eq:rosenbrocklin}
		S(z) = \begin{bmatrix} A-z I_r & B \\ C & A_0+A_1z \end{bmatrix},
	\end{equation}
	where the coefficient matrices $\{A,B,C, A_0,A_1\}$ are randomly generated
	(using the MATLAB code \texttt{M = randn(m,n)+1i*randn(m,n)})
	and the dimensions are set to $n=100$ and $r=10$. 
	We compute the backward error $\eta(\lambda)$ with full
	perturbation using the formula~\eqref{eq:allbloexp}.
	For testing, we generate the approximate eigenvalue $\lambda$ 
	by first perturbing the coefficient matrices of~\eqref{eq:rosenbrocklin} 
	entrywisely by \texttt{randn*1.0E-1},
	and then choosing $\lambda$ as an eigenvalue of the perturbed linear system.
	We thus anticipate the backward error of such $\lambda$ to be at the same 
	order of the perturbation $\mathcal O(10^{-1})$.

	\begin{figure}
		\centering
		\includegraphics[width=0.48\textwidth]{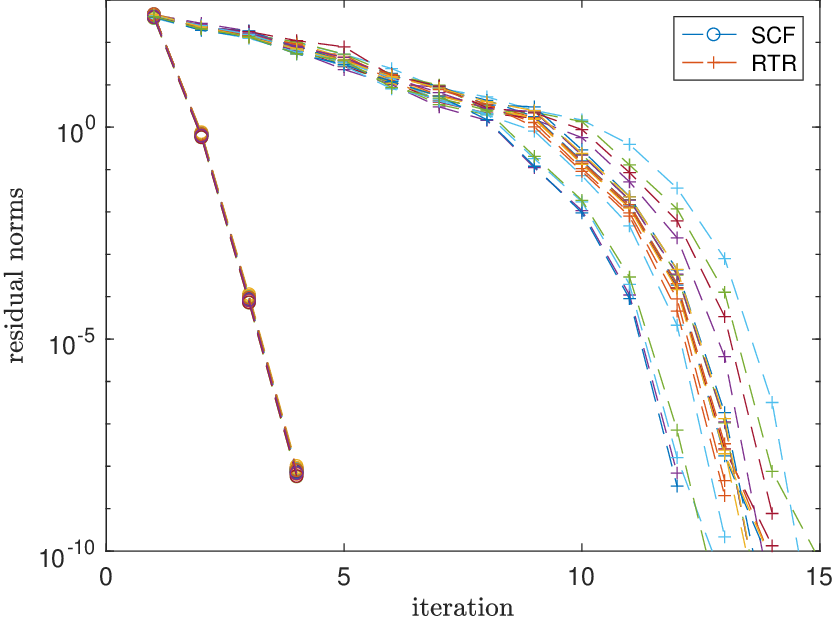}~\includegraphics[width=0.48\textwidth]{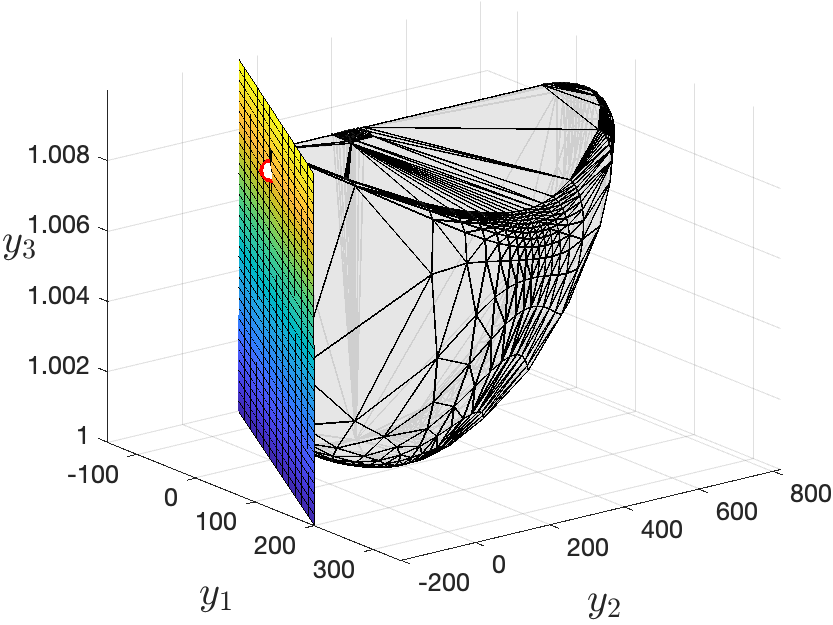}
		\caption{
			Rosenbrock system~\eqref{eq:rosenbrocklin} with $n=100$ and $r=10$
			to compute the backward error $\eta(\lambda)$ by~\eqref{eq:allbloexp}.
			{\bf Left}: Convergence history of residual norms 
			$\| H(x_k) x_k - s_k x_k\|_2$ with $s_k = x_k^*H(x_k)x_k$ for
			iterative $x_k$ by level-shifted SCF~\eqref{eq:lsscf} and Riemannian Trust-Region method. 
			Reported are $20$ repeated runs with randomly generated starting vectors $x_0$. 
			{\bf Right}: Illustration of the corresponding JNR minimization~\eqref{eq:optw},
			with approximate joint numerical range $W(\mathcal M)$
			(gray surface) and level surface of $g(y):=y_1 + y_2/y_3$ (color surface) 
			at the solution $y_{\star}=\rhom(x_{\star})$ (marked `o').
		} \label{fig:backerr}
	\end{figure}

	Recall that the formula of the backward error in~\eqref{eq:allbloexp} involves an SRQ2
	minimization in the form of~\eqref{eq:srq2s}
	-- with coefficient matrices $A_1=G_1$, $A_2=G_2$, and $A_3=(H_1+\gamma
	H_2)$ from~\eqref{eq:gmat}.
	The left panel of~\Cref{fig:backerr} reports the convergence history of 
	the Riemannian trust-region method for solving this optimization,
	as well as the level-shifted SCF~\eqref{eq:lsscf} for the corresponding
	NEPv~\eqref{eq:nepv}.
	We have repeatedly run the algorithms $20$ times with different and 
	randomly generated starting vectors $x_0$. 
	In all testing cases, SCF rapidly converged in three iterations,
	despite of a seemingly linear rate of convergence, 
	and the algorithm is insensitive to the choice of initial vectors.
	In comparison, the RTR requires more iterations
	-- although its local convergence rate seems superlinear,
	the initial convergence is relatively slow. 
	We note that the total number of iterations by RTR is about $4$ times that of
	SCF, but SCF is about $10$ times as fast in computation time:
	\[
	\mbox{SCF = $0.05$ seconds v.s. RTR = $0.57$ seconds}.
	\]
	This indicates the average cost per iteration of RTR is more expensive than
	that of SCF.
	The reason is mainly because RTR required a large number of {\em inner
	iterations} to solve the trust-region subproblem at each (outer)
	iteration~\cite{Absil:2009}.
	Taking inner iterations into account, RTR requires 
	an average of 176 iterations,
	and level-shifted SCF still requires $3$ 
	(i.e., $\sigma_k=0$ are always accepted and the process is indeed a plain SCF).

	For all the $20$ runs with different starting vectors,
	both RTR and SCF have converged to the same solution,
	yielding the backward error 
	$\eta(\lambda) \approx   0.0783127\dots,$
	which is indeed at the same order of the perturbation $\mathcal O(10^{-1})$.
	The global optimality of the computed solution is verified pictorially
	in the right panel of~\Cref{fig:backerr} through the corresponding 
	JNR minimization in the form of~\eqref{eq:eg1b}.
} %
\end{example}

\begin{example}\label{eg:blocks}
{\rm 
	This example is to demonstrate the computational efficiency of NEPv approaches
	for various SRQ2 minimization~\eqref{eq:srq2min} from~\Cref{preliminaries}
	for the computation of backward error of Rosenbrock systems with different types of block perturbations.
	We will focus on $7$ types of block perturbation, including 
	`AP', `BC', `ABC', `ABP', `ACP', `BCP', and `ABCP',
	where the letters indicate the perturbing blocks in the Rosenbrock
	system matrix~\eqref{eq:rosenbrock}.
	The formulas of backward error for each of those perturbation type 
	are found in~\Cref{preliminaries}, and they all involve SRQ2
	minimization~\eqref{eq:srq2min}.
    Note that we have discussed a total of $15$ cases of partial block perturbations
    in~\Cref{preliminaries}, among which $8$ cases can be conveniently solved by
    Hermitian eigenvalue problems and, hence, are not reported in our experiments.

	\begin{table}
		\begin{center}
		\caption{
			\Cref{eg:blocks}: Rosenbrock system matrix~\eqref{eq:rosenbrocklin} of
			dimensions $r=10$ and $n=200,400,600,800$. 
		}
		\label{tab:egn}
		\small\addtolength{\tabcolsep}{-3pt}
		\begin{tabular}{cc|rrrrrrr} \toprule 
			&     & \multicolumn{7}{c}{ `timing (\# iteration)' per type of block perturbation}\\ 
			$n$ & Alg & AP~~~ &BC~~~ &ABC~~& ABP~~& ACP~~ & BCP~~ & ABCP~~ \\ \midrule 
			\multirow{2}{*}{$200$}& SCF  &  0.4 ( 7)&  1.1 (48) &  0.8 (38) & 0.2 ( 6) &  0.1 ( 5) & 0.1 ( 5) & 0.1 ( 3)\\ 
			& RTR  &  4.2 (43)&  3.3 (19) &  3.8 (26) & 1.8 (14) &  6.3 (46) & 1.9 (14) & 1.8 (14)\\ \hline
			\multirow{2}{*}{$400$}& SCF  &  1.7 ( 6)&  1.0 (14) &  1.0 (14) & 0.3 ( 4) &  0.2 ( 3) & 0.5 ( 3) & 0.2 ( 3)\\ 
			& RTR  & 22.0 (43)& 10.7 (26) & 10.3 (35) & 5.5 (15) & 18.6 (39) & 7.0 (15) &33.5 (29)\\ \hline
			\multirow{2}{*}{$600$}& SCF  &  4.9 ( 5)&  2.0 (12) &  2.2 (13) & 0.9 ( 4) &  0.5 ( 3) & 0.5 ( 3) & 0.4 ( 2)\\ 
			& RTR  & 52.2 (80)& 38.3 (24) & 29.0 (31) &24.8 (19) & 88.1 (70) &28.9 (17) &25.5 (18)\\ \hline
			\multirow{2}{*}{$800$}& SCF  & 18.7 ( 5)&  4.4 (14) &  4.7 (15) & 1.2 ( 4) &  1.2 ( 3) & 0.9 ( 3) & 0.6 ( 2)\\ 
			& RTR  &177.9 (76)& 61.1 (22) & 63.4 (37) &44.4 (17) &118.8 (62) &52.9 (16) &43.6 (17)\\ \hline
		\end{tabular}
		\end{center}
	\end{table}

	\begin{table}
		\begin{center}
		\caption{
			\Cref{eg:blocks}: Rosenbrock system matrix~\eqref{eq:rosenbrocklin} of
			dimensions $n=10$ and $r=200,400,600,800$. 
		}
		\label{tab:egr}
		\small\addtolength{\tabcolsep}{-3pt}
		\begin{tabular}{cc|rrrrrrr} \toprule 
			&     & \multicolumn{7}{c}{ `timing (\# iteration)' per type of block perturbation}\\ 
			$r$ & Alg & AP~~~ &BC~~~ &ABC~~  & ABP~~& ACP~~ & BCP~~ & ABCP~~ \\ \midrule 
			\multirow{2}{*}{$200$}& SCF	 &  0.9 ( 9)&  0.2 (12) &  0.2 ( 5) & 0.9 ( 8) &  0.1 ( 6) & 0.2 (12) & 0.1 ( 5)	\\ 
			& RTR	 &  3.5 (14)&  2.7 (12) &  1.8 (11) & 2.8 (17) &  2.3 (11) & 2.7 (13) & 1.8 (11)\\ \hline
			\multirow{2}{*}{$400$}& SCF	 &  8.5 ( 9)&  0.7 ( 9) &  0.3 ( 4) & 0.6 ( 7) &  0.4 ( 5) & 0.7 ( 9) & 0.3 ( 4)	 \\ 
			& RTR	 & 19.5 (21)& 15.9 (14) &  9.1 (12) &23.5 (25) & 10.9 (12) &12.9 (13) & 9.2 (12)\\ \hline
			\multirow{2}{*}{$600$}& SCF	 & 18.6 ( 8)&  1.3 ( 7) &  0.9 ( 3) &16.2 ( 7) &  1.2 ( 5) & 1.8 ( 8) & 0.8 ( 4)	 \\ 
			& RTR	 & 40.7 (17)& 30.8 (12) & 16.7 (10) &27.0 (16) & 18.7 (10) &34.1 (12) &15.7 (10)\\ \hline
			\multirow{2}{*}{$800$}& SCF	 & 41.5 ( 7)&  2.5 ( 8) &  0.9 ( 3) &31.7 ( 6) &  1.9 ( 4) & 2.8 ( 9) & 1.6 ( 4)	 \\ 
			& RTR	 & 77.8 (22)& 59.1 (12) & 36.5 (12) &57.0 (21) & 33.9 (11) &68.0 (13) &35.3 (12)\\ \hline
		\end{tabular}
		\end{center}
	\end{table}
	%

	\begin{table}
		\begin{center}
		\caption{
			\Cref{eg:blocks}: Rosenbrock system matrix~\eqref{eq:rosenbrocklin} of
			dimensions $r=n$ and $n=200,400,600,800$. 
		}
		\label{tab:egnr}
		\small\addtolength{\tabcolsep}{-4pt}
		\begin{tabular}{cc|rrrrrrr} \toprule 
			&     & \multicolumn{7}{c}{ `timing (\# iteration)' per type of block perturbation}\\ 
			$n$ &Alg & AP~~~ &BC~~~ &ABC~~  & ABP~~& ACP~~ & BCP~~ & ABCP~~ \\ \midrule 
			\multirow{2}{*}{$200$}& SCF	 &  0.8 ( 8)&  0.8 ( 8) &  0.5 ( 7) &  0.5 ( 7) &  0.3 ( 5) &  0.3 ( 5) &  0.2 ( 3)	\\ 
			& RTR	 &  6.1 (14)&  8.2 (14) &  9.4 (15) &  7.3 (15) &  9.1 (15) &  8.5 (14) &  6.0 (14)\\ \hline
			\multirow{2}{*}{$400$}& SCF	 &  2.5 ( 8)&  2.5 ( 8) &  2.4 ( 7) &  2.9 ( 7) &  1.6 ( 5) &  1.6 ( 5) &  0.9 ( 3)	 \\ 
			& RTR	 & 40.3 (15)& 49.7 (16) & 41.0 (16) & 45.8 (16) & 59.5 (16) & 57.0 (16) & 41.5 (16)\\ \hline
			\multirow{2}{*}{$600$}& SCF	 &  7.2 ( 8)&  6.9 ( 8) &  6.3 ( 7) &  6.2 ( 7) &  4.0 ( 5) &  4.0 ( 5) &  3.5 ( 4)	 \\ 
			& RTR	 &182.0 (18)&167.4 (18) &167.7 (18) &136.5 (17) &224.3 (18) &131.0 (17) &130.7 (17)\\ \hline
			\multirow{2}{*}{$800$}& SCF	 & 12.8 ( 8)& 14.9 ( 8) & 11.4 ( 7) & 11.2 ( 7) &  8.3 ( 5) &  8.0 ( 5) &  5.6 ( 3)	 \\ 
			& RTR	 &284.4 (16)&396.5 (20) &293.8 (18) &263.9 (17) &383.5 (17) &429.5 (19) &217.8 (16)\\ \hline
		\end{tabular}
		\end{center}
	\end{table}
	%
	%
	
	For experiment, we still use the Rosenbrock system matrix~\eqref{eq:rosenbrocklin}
	from~\Cref{eg:linsys}.
	We generate such $S(z)$ with various dimensions $r$ and $n$ 
	and compute the eigenvalue backward errors for different 
	types of block perturbations via SRQ2 minimization~\eqref{eq:srq2min}.
	It is verified that condition~\eqref{eq:comnull} holds for all the 
	SRQ2 minimizations, so their optimal solutions are characterized by
	the NEPv~\eqref{eq:nepv} due to~\Cref{thm:nondiff}.
	In the test, we compare the performance of 
	Riemannian trust-region method for directly solving SRQ2 minimization~\eqref{eq:srq2min}
	and the level-shifted SCF for the corresponding NEPv~\eqref{eq:nepv}.
	The starting vectors of the algorithms are set to be the minimizer 
	of an individual Rayleigh quotient in~\eqref{eq:srq2min}, whichever bears a smaller value.  
	In all testing cases, the two algorithms have converged to the same solution. 
	\Cref{tab:egn,tab:egr,tab:egnr} report the computation time in seconds and number of
	iterations (as marked in parentheses) of the algorithms.
	We can see that the NEPv approach is always faster,
	and its save in computation time is quite remarkable as the dimension of problems increase.
	In a few cases, despite SCF takes more iterations than RTR,
	it can still run faster -- as we have also observed previously
	in~\Cref{eg:linsys}.
	This again indicates that each SCF step is cheaper than RTR's,
	thanks to the use of state-of-the-art eigensolvers by SCF.
} %
\end{example}

\section{Concluding Remarks}\label{sec:conclusion}
We have derived computable formulas for the structured eigenvalue backward error
of the Rosenbrock system matrix, considering both full and partial block
perturbations. 
These formulas are unified under a class of SRQ2 minimization problems, which
involve minimizing the sum of two generalized Rayleigh quotients. 
We demonstrated that these optimization problems can be reformulated
to the optimization of a rational function over the joint numerical range of
three Hermitian matrices. 
This reformulation helps to avoid certain local minima in the original problem
and to visualize the optimal solution.
Additionally, by exploiting the convexity in the joint numerical range, we
established an NEPv characterization for the optimal solution.
The effectiveness of our NEPv approach was also illustrated through numerical
examples.

\bibliographystyle{siam}
\bibliography{PraS22b.bib,refsDL.bib}

\end{document}